\documentclass[11pt,bibliography=totoc,reqno]{scrartcl}

\usepackage[ansinew]{inputenc}
\usepackage[T1]{fontenc}
\usepackage[english]{babel}
\usepackage[babel]{csquotes} %hinzugefügt, weil Warnung bei englischer Übersetzung kompilierung
\usepackage{marvosym}

\usepackage{caption}
\usepackage{todonotes}
\usepackage{tikz-3dplot}
\usetikzlibrary{calc,3d,arrows}
\usepgflibrary{arrows}

\usepackage[backend=bibtex8,style=alphabetic]{biblatex} %bibtex zu bibtex8 gemacht, wegen Warnung bei englishkompilierung

\usepackage[style=alphabetic]{biblatex}
\usepackage{units}
\usepackage{ifsym}
\usepackage[intlimits,leqno]{amsmath}
\usepackage{amsthm}
\usepackage{amssymb}
\usepackage{amsfonts}
\usepackage{dsfont}
\usepackage{exscale}
\usepackage{txfonts}
\usepackage{pxfonts}
\usepackage{bbold}
\usepackage{bbm}
\usepackage{graphicx}
\usepackage{wrapfig}
\usepackage{wasysym}
\usepackage{mathrsfs}
\usepackage{setspace}
\usepackage{scrpage2} %für Kopfzeile
\usepackage{todonotes}
\usepackage[all]{xy}
\usepackage{tikz,pgf}
\usepackage{float}
\usetikzlibrary{matrix,arrows}

\bibliography{literatur}

\newcommand{\e}{\operatorname{e}}
\newcommand{\rd}{\mathrm{d}}
\newcommand{\re}{\operatorname{Re}}
\newcommand{\im}{\operatorname{Im}}

\newcommand{\I}{\mathrm{i}}

\DeclareMathOperator{\Span}{span}
\DeclareMathOperator{\tr}{tr}

\newtheorem{theorem}{Theorem}%[section]
\swapnumbers %Nummer ist vorne
\newtheorem{remark}{Remark}[section]
\newtheorem{lemma}[remark]{Lemma}%[section]

\theoremstyle{definition}
\newtheorem{definition}[remark]{Definition}%[section]
\newtheorem{example}[remark]{Example}%[section]

\automark[subsection]{section}

\usepackage{multirow}

\begin{document}

\title{Lattices of oscillator groups}
\author{Mathias Fischer}
\maketitle

\begin{abstract}
This paper is concerned with discrete, uniform subgroups (lattices) of oscillator groups, which are certain semidirect products of the Heisenberg group and the additive group $\mathds{R}$ of real numbers.

The present paper rectifies the uncertainties in \cite{MR85} of Medina and Revoy and gives a complete classification of the lattices of the $4$-dimensional oscillator group up to isomorphism.

%An oscillator group is a semidirect product of the Heisenberg group and the additive group $\mathds{R}$ of real numbers with respect to a one-parameter group of the automorphisms of the Heisenberg group, satisfying certain conditions.

%They are the only groups, beside $SL(n,\mathds{R})$ and abelian ones, which have an biinvariant Lorentzian metric. In addition, the Quotient of an oscillator group by a lattice gives an example of a compact homogeneous lorentzian manifold.
%
%In this paper we classify the lattices of the $4$-dimensional oscillator group. Furthermore, we give an idea for the classification of the lattices of the oscillator group of arbitrary dimension (Theorem \ref{theoremLCapHeis=GammaR}).
%
%There is already a try of Medina and Revoy in \cite{MR85} to classify the lattices of the oscillator group. But the given maps in \cite[p. 92]{MR85} aren't necessarily automorphisms, different from the assumption.
%Hence, théorème III and the following corollaire is incorrect.
%
%The current paper rectifies the uncertainties and gives a complete classification of the lattices of the $4$-dimensional oscillator group (Theorem \ref{Klassifik2} and Theorem \ref{Klassifik}).
\end{abstract}

\section{Introduction}

Oscillator groups are certain semidirect products of the Heisenberg group and the additive group $\mathds{R}$. They are interesting because they are the only simply connected solvable Lie groups, besides the abelian ones, which have a biinvariant Lorentzian metric. In addition, the quotient of an oscillator group by a lattice gives an example of a compact homogeneous Lorentzian manifold.
%There are oscillator groups of the same dimension (except dimension $4$), which are not isomorphic.

In \cite{MR85}, Medina and Revoy already classified the lattices of the oscillator group. But unfortunately the given maps in \cite[p. 92]{MR85} aren't necessarily automorphisms, different from the assumption.
Hence, theorem III and the following corollary is incorrect.

Our goal is to classify the lattices of the $4$-dimensional oscillator group.

More precisely we consider groups $Osc_n(\omega,B)$, i. e. $\mathds{R}\times\mathds{R}^{2n}\times\mathds{R}$ with the group multiplication given by \[(z,\xi,t)(v,\eta,s)=\left(z+v+\frac{1}{2}\omega(\xi,\e^{tB}\eta),\xi+\e^{tB}\eta,t+s\right).\]
Here $B$ is an invertible $2n\times 2n$-matrix and $\omega$ a symplectic form on $\mathds{R}^{2n}$ such that $\omega(B\xi,\eta)=-\omega(\xi,B\eta)$ for all $\xi,\eta\in\mathds{R}^{2n}$ and $\omega(B\cdot,\cdot)$ is definite. Every oscillator group is isomorphic to some $Osc_n(\omega,B)$. In addition, these groups are isomorphic to the groups $G_k(\lambda)$ considered in \cite{MR85}.

We compute the automorphisms of $Osc_n(\omega,B)$ in Theorem \ref{theoremAutOsc}. The theorem also rectifies the assertion in \cite[p. 92]{MR85}.
Afterwards we classify the lattices of $Osc_1(\omega,B)$ in three steps  (Theorem \ref{theoremLCapHeis=GammaR}-\ref{Klassifik}).
In Theorem \ref{theoremLCapHeis=GammaR}, which also holds for lattices in $Osc_n(\omega,B)$, we get to know that we can always assume 
that a lattice of a $4$-dimensional oscillator group is generated by $\{(1,0,0),(0,e_i,0)\mid i=1,2\}$ and some 4th element $(0,\xi_0,1)$ in an oscillator group $Osc_1(\omega_r,B)$ with a certain standard symplectic form $\omega_r$.
Furthermore, we see in Theorem \ref{Klassifik2} that there is also, besides $\omega_r$, a unique standard matrix, denoted by $B=\lambda B_{x,y}$.
Then, the last step is to describe all lattices of the special kind we get in Theorem \ref{Klassifik2} up to automorphisms in Theorem \ref{Klassifik}. This gives restrictions for $\xi_0$.
Our main tools for these steps are special isomorphisms and automorphisms of oscillator groups preserving the subgroup $\langle(1,0,0),(0,e_i,0)\mid i=1,2\rangle$, we characterize in Lemma \ref{lemmaIsoOscMitGitter} and Lemma \ref{Autos}.

%In Theorem \ref{theoremLCapHeis=GammaR} we compute a certain standard symplectic form $\omega_r$ for each lattice. 
%Then we see in Theorem \ref{Klassifik2} that there is also a unique standard matrix, denoted by $B=\lambda B_{x,y}$.
%The last step is to describe all lattices of the special kind we get in Theorem \ref{Klassifik2} up to automorphisms.
%For these steps we use special isomorphisms respectively automorphisms, we characterize in Lemma \ref{lemmaIsoOscMitGitter} and Lemma \ref{Autos}.

Finally, each lattice $L$ in $Osc_1(\omega,B)$ gives a data $(r,\lambda,(x,y),\xi_0)$.
Thereby $r$ is a positive integer, $\lambda$ a certain angle, $(x,y)$ an element of the "half fundamental domain" for the modular group and $\xi_0$ a vector from a finite set to be extracted from the list in Appendix \ref{AppA}.

This data describes the lattices of the oscillator group in the sense that two lattices $L_1$ and $L_2$ of $Osc_1(\omega,B)$ have the same data if and only if there is an automorphism in $Osc_1(\omega,B)$ mapping $L_1$ onto $L_2$.

%Finally we get for each lattice $L$ in $Osc_1(\omega,B)$ a uniquely determined symplectic form $\omega_r$ and $\lambda B_{x,y}\in GL(2,\mathds{R})$, which satisfy certain conditions, a uniquely determined $\xi_0\in\mathds{R}^{2}$ to extract from the list in \ref{AppA} and an isomorphism $\varphi:Osc_1(\omega,B)\rightarrow Osc_1(\omega_r,\lambda B_{x,y})$, such that $\varphi(L)=\langle(1,0,0),(0,e_i,0),(0,\xi_0,1)\mid i=1,2\rangle$.

%At that, for a subset $A$ of $Osc_n(\omega,B)$ let $\langle A\rangle$ denote the subgroup of $Osc_n(\omega,B)$ generated by $A$.

%After giving some definitions and first results, we compute the automorphism group of the oscillator group in Theorem \ref{theoremAutOsc}. The theorem also shows that the given maps in \cite[p. 92]{MR85} aren't necessarily automorphisms, different from the assumption.
%In section \ref{Sectionlattices}, we begin with showing that every lattice of an oscillator group has a relatively ``nice'' structure in an oscillator group with a certain standard Heisenberg group (Theorem \ref{theoremLCapHeis=GammaR}).
%Using the Lemma \ref{lemmaIsoOscMitGitter} we construct a certain standard oscillator group for each lattice of an $4$-dimensional oscillator group in Theorem \ref{Klassifik2}.
%Afterward, we compute special automorphisms in Lemma \ref{Autos} for further observation.
%Finally, Theorem \ref{Klassifik} gives a complete classification of the lattices of the $4$-dimensional oscillator group.

\section{Oscillator groups}

\begin{definition}\label{DefOscGr}
For a symplectic form $\omega$ on $\mathds{R}^{2n}$, let $H_n(\omega)$ denote the group $\mathds{R}\times\mathds{R}^{2n}$ with the multiplication given by \[(z,\xi)(v,\eta)=\big(z+v+\frac{1}{2}\omega(\xi,\eta),\xi+\eta\big).\]
Since each of these groups are isomorphic we call them Heisenberg groups.\\
Let $H$ be a Heisenberg group, $\mathfrak{h}$ its Lie algebra and $\mathfrak{z}$ the center of $\mathfrak{h}$. Suppose that $p$ is a one-parameter subgroup of the automorphism group of $H$ with trivial action on the center of the Heisenberg group and satisfying that the map
\begin{align*}\label{A}
A:\mathfrak{h}/\mathfrak{z}\times\mathfrak{h}/\mathfrak{z}\rightarrow\mathfrak{z}\cong\mathds{R},\quad A(h_1+\mathfrak{z},h_2+\mathfrak{z}):=\big[((\rd_0f)(1))(h_1),h_2\big]
\end{align*}
is definite, where $f:\mathds{R}\rightarrow Aut(\mathfrak{h})$ is the differential of the automorphism $p(t)\in Aut(H)$ at the point $0\in\mathfrak{h}$, i. e. $f(t):=\rd_0(p(t))$.
Then the semidirect product $H\rtimes_p\mathds{R}$ of $H$ and $\mathds{R}$ with respect to $p$ is called oscillator group.
\end{definition}

Later, we will see that these groups are isomorphic to the groups $G_k(\lambda)$ introduced in \cite{MR85}.

\begin{lemma}\label{ficken} 
Let $p$ be a one-parameter group of the automorphisms of $H_n(\omega)$. Then $p$ satisfies the conditions in Definition \ref{DefOscGr} if and only if there is a $\delta\in\mathds{R}^{2n}$ and a $B\in\mathfrak{gl}(2n,\mathds{R})$ satisfying $\omega(B\xi,\eta)=-\omega(\xi,B\eta)$ for all $\xi,\eta\in\mathds{R}^{2n}$
    and $\omega(B\cdot,\cdot)$ is %symmetric and 
					definite, such that 
               \[p(t)=\exp\left(t\begin{pmatrix}
                                0&\delta^T\\
                         		  0&B
                      \end{pmatrix}\right).\]
If $B$ satisfies $\omega(B\xi,\eta)=-\omega(\xi,B\eta)$ for all $\xi,\eta\in\mathds{R}^{2n}$
    and $\omega(B\cdot,\cdot)$ is definite, then $B$ has purely imaginary eigenvalues and can be diagonalized over $\mathds{C}$.
\end{lemma}

\begin{proof}
The equivalence follows directly from the definition and known results for automorphisms of the Heisenberg group (see for instance \cite{To78}).\\
Let $\omega(B\cdot,\cdot)$ be positive definite (otherwise we consider -$\omega(B\cdot,\cdot)$) and consider the complexification of $B$ and the complex bilinear extension of $\omega$.
One can check that $\omega(Bz,\overline{z})\in\mathds{R}$ and $\omega(Bz,\overline z)>0$ for all $z\in \mathds{C}^{2n}$.
Suppose $z$ is an eigenvector of $B$ with its corresponding eigenvalue $\lambda\neq 0$, then we see $\mathds{R}\ni\omega(Bz,\overline z)=\lambda\omega(z,\overline z)$.
On the other hand $\omega(z,\overline z)\in\I\mathds{R}$. Hence, $\lambda\in\I\mathds{R}$ and finally the first part follows.\\
Now, we prove the second part by induction over $n$. For $n=1$ the assertion holds.
Assuming the assertion holds for $n$, we will prove it for $n+1$.
Let $Z_1$ be an eigenvector of $B$ over $\mathds{C}$ with corresponding eigenvalue $\I\lambda_1$, $\lambda\neq 0$.
Then $\omega(Z_1,\overline{Z_1})=\frac{1}{\I\lambda}\omega(BZ_1,\overline{Z_1})\neq 0$. In particular $\omega|_{\Span\left\{Z_1,\overline{Z_1}\right\}}$ is nondegenerate.
Also $\overline{Z_1}$ is an eigenvector of $B$ with corresponding eigenvalue $-\I\lambda_1$, so $B$ maps the subspace $\Span\left\{Z_1,\overline{Z_1}\right\}$ onto itself. 
Since $B$ is antisymmetric with respect to $\omega$, $B$ maps $\Span\left\{Z_1,\overline{Z_1}\right\}^\perp$ into itself. 
Making use of the induction hypothesis on $B\big|_{\Span\left\{Z_1,\overline{Z_1}\right\}^\perp}$ yields the assertion.
\end{proof}

\begin{remark}
The Lie algebra of the oscillator group $H_n(\omega)\rtimes_{p}\mathds{R}$ is the semidirect sum of $\mathfrak{h}_n(\omega)$ and $\mathds{R}$ with respect to the derivation
$s \mapsto s\left(\begin{smallmatrix}
					0 & \delta^T \\
  	      	   0 & B
   	   	 \end{smallmatrix}\right).$\\
Hence, there is a basis $\left\{X_1,Y_1,\dots,X_n,Y_n\right\}$ of $\mathds{R}^{2n}$ satisfying $\omega(X_i,Y_j)=\delta_{i,j}$, ($\delta_{i,j}$ denotes the Kronecker symbol) and $\omega(X_i,X_j)=\omega(Y_i,Y_j)=0$ for i,j=1,\dots,n such that the Lie algebra of the oscillator group $H_n(\omega)\rtimes_{p}\mathds{R}$ is \[\mathds{R}Z\oplus\Span\left\{ X_1,Y_1,\dots,X_n,Y_n\right\}\oplus\mathds{R}T\]
with the non-zero brackets of $\big\{Z,X_1,Y_1,\dots,X_n,Y_n,T\big\}$ given by
\[[X_i,Y_i]=Z,\quad [T,X_i]=BX_i+\delta_{2i-1}Z,\quad [T,Y_i]=BY_i+\delta_{2i}Z\]
for $i=1,$ $\dots,$ $n.$ Here $\delta_i$ denotes the $i$-th. component of $\delta$.
\end{remark}

\begin{definition}
The Lie group $Osc_n(\omega,B)$ is the oscillator group $H_n(\omega)\rtimes_p\mathds{R}$, where 
$p(t)=\left(\begin{smallmatrix}
                 1    &   0  \\
                 0    &   \e^{tB}
\end{smallmatrix}\right)$, $\omega(B\cdot,\cdot)=-\omega(\cdot,B\cdot)$ and $\omega(B\cdot,\cdot)$ definite.
\end{definition}

\begin{remark}
Each oscillator group $H_n(\omega)\rtimes_{p}\mathds{R}$ is isomorphic to some $Osc_n(\omega,B)$, since the map
\[\phi: T\mapsto T+\sum_{j=1}^n(\delta_{2j}X_j-\delta_{2j-1}Y_j),\quad
    X_i\mapsto X_i,\quad    Y_i\mapsto Y_i,\quad    Z\mapsto Z\]
is an isomorphism from the Lie algebra of $H_n(\omega)\rtimes_{p}\mathds{R}$ to $\mathfrak{osc}_n(\omega,B)$, where 
$p(t)=\exp\left(t\left(\begin{smallmatrix}
                         0&\delta^T\\
                         0&B
                      \end{smallmatrix}\right)\right)$, $\omega(B\cdot,\cdot)=-\omega(\cdot,B\cdot)$ and $\omega(B\cdot,\cdot)$ definite.
I. e., we can always assume that $\delta=0$.
\end{remark}

The group multiplication in $Osc_n(\omega,B)$ is given by
\[(z,\xi,t)(v,\eta,s)=\left(z+v+\frac{1}{2}\omega(\xi,\e^{tB}\eta),\xi+\e^{tB}\eta,t+s\right),\]
and inversion by \[(z,\xi,t)^{-1}=(-z,-\e^{-tB}\xi,-t),\]
where $\xi,\eta\in\mathds{R}^{2n}$ and $z,v,s,t\in\mathds{R}$.

\begin{definition}
\newcommand*{\tempb}{\multicolumn{1}{|c}{}}
\newcommand*{\temp}{\multicolumn{1}{c|}{}}
We also define for $\lambda=(\lambda_1,\dots,\lambda_n)$, $\lambda_i\in\mathds{R}$ the matrix \[N_{\lambda}:=\left(\begin{array}{ccc} 
\footnotesize \begin{array}{rc} 0 & -\lambda_1 \\ \lambda_1 & 0\end{array} & \tempb & \mathbf{0}\\ \cline{1-1}
 & \ddots & \\ \cline{3-3}
\mathbf{0} & \temp & \footnotesize \begin{array}{rc} 0 & -\lambda_n \\ \lambda_n & 0\end{array}
\end{array}\right)\]
and the symplectic form $\omega_{\lambda}(\xi,\eta):=\xi^TN_{-\lambda}\eta$.
We denote $Osc_n(\omega_{(1,\dots,1)},N_{\lambda})$ by \linebreak $Osc(\lambda_1,\dots,\lambda_n)$, where $\lambda_i\lambda_{i+1}>0$ for $i=1,\dots n-1$.
\end{definition}

\begin{lemma}\label{Iso00}
The oscillator groups $Osc_1(\omega,B)$ and $Osc(\lambda_1,\dots,\lambda_n)$, where $0<\lambda_1\leq\dots\leq\lambda_n$ denote the positive imaginary parts of the eigenvalues of $B$ with multiplicity, are isomorphic.
\end{lemma}

\begin{remark}
Note that $Osc(\lambda_1,\dots,\lambda_n)$ is isomorphic to $G_k(\lambda)$ in \cite{MR85}.
\end{remark}

\begin{proof}[Proof of the lemma]
Suppose that $\omega(B\cdot,\cdot)$ is positive definite
(similarly for $\omega(B\cdot,\cdot)$ negative definite).\linebreak
Let $\left\{Z_1,\overline{Z_1},\dots,Z_n,\overline{Z_n}\right\}$ be a basis of eigenvectors as in the proof of Lemma \ref{ficken}, such that $BZ_j=\I\lambda_jZ_j$ with $\lambda_j>0$ and $\lambda_1\leq\dots\leq\lambda_n$. We set $\mu_j:=\frac{\I}{2}\omega(Z_j,\overline{Z_j})=\frac{\I}{2\lambda_j}\omega(BZ_j,\overline{Z_j})>0$.
Then \[Z\mapsto Z,\quad\frac{1}{\sqrt{\mu_j}}\re(Z_j)\mapsto X_j,\quad \frac{1}{\sqrt{\mu_j}}\im(Z_j)\mapsto Y_j,\quad T\mapsto -T\] is an isomorphism from $\mathfrak{osc}_n(\omega,B)$ to $\mathfrak{osc}(\lambda_1,\dots,\lambda_n)$.
%Because of the basis \[\big\{\frac{1}{\sqrt{\mu_1}}\re(X_1),\frac{1}{\sqrt{\mu_1}}\im(X_1),\dots,\frac{1}{\sqrt{\mu_n}}\re(X_n),\frac{1}{\sqrt{\mu_n}}\im(X_n)\big\}\] of the Lie algebra of $Osc_1(\omega,B)$ we see that $Osc_n(\omega,B)$ and $Osc(-\lambda_1,\dots,-\lambda_n)$ are isomorphic.
%Finally the assertion follows from the fact that $\psi:(z,\xi,t)\mapsto(z,\xi,-t)$ is an isomorphism from $Osc(-\lambda_1,\dots,-\lambda_n)$ to $Osc(\lambda_1,\dots,\lambda_n)$.
\end{proof}

\begin{lemma} \label{IsoO} (Medina/Revoy: \cite[p. 91]{MR85})\\
Let $0<\lambda_1\leq\dots\leq\lambda_n$ and $0<\lambda'_1\leq\dots\leq\lambda'_n$. Then $Osc(\lambda_1,\dots,\lambda_n)$ and $Osc(\lambda'_1,\dots,\lambda'_n)$ are isomorphic if and only if there is a $k\in\mathds{R}\backslash\{0\}$, such that $k\lambda'_i=\lambda_i$ for $i=1,\dots,n.$
\end{lemma}

\begin{proof}
Let $\mathfrak{osc}$ denote the Lie algebra of $Osc(\lambda_1,\dots,\lambda_n)$.
We choose $b+[\mathfrak{osc},\mathfrak{osc}]\in \nicefrac{\mathfrak{osc}}{[\mathfrak{osc},\mathfrak{osc}]},$ $b\notin [\mathfrak{osc},\mathfrak{osc}]$, and assign to this element a linear operator $\hat{B}$ of $\nicefrac{[\mathfrak{osc},\mathfrak{osc}]}{\mathfrak{Z}}$, by
\begin{align*}
\hat{B}(x+\mathfrak{Z})&=[b,x]+\mathfrak{Z}.
\end{align*}
The operator $\hat{B}$ and therefore also its eigenvalues are uniquely determined by the structure of the Lie algebra up to a factor $k\neq 0$.
On the other hand, eigenvalues of $\hat{B}$ associated with $b=(0,0,1)^T$ are exactly the eigenvalues of $N_\lambda$.

Hence, if there is no $k\in\mathds{R}\backslash\{0\}$, such that $\lambda_i=k\lambda'_i$ for $i=1,\dots,n,$, then $Osc(\lambda_1,\dots,\lambda_n)$ and $Osc(\lambda'_1,\dots,\lambda'_n)$ are not isomorphic.

For the other implication note that the map
\[\varphi:X_i\mapsto X_i,\,Y_i\mapsto Y_i,\,Z\mapsto Z,\,T\mapsto kT\]
is an isomorphism of Lie algebras from $\mathfrak{osc}(\lambda_1,\dots,\lambda_n)$ to $\mathfrak{osc}(\lambda'_1,\dots,\lambda'_n)$, where \linebreak $\left\{X_1,Y_1,\dots,X_n,Y_n,Z,T\right\}$ is the standard basis of $\mathds{R}^{2n+2}$.
\end{proof}

Thus, each oscillator group of dimension $2n+2=4$ is isomorphic to $Osc(1)$.

\begin{theorem} \label{theoremAutOsc}
A map $\phiup:Osc_n(\omega, B)\rightarrow Osc_n(\omega, B)$ is an automorphism if and only if there are numbers $m\in\mathds{R}$, $\mu\in\left\{+1,-1\right\}$ and $a\in\mathds{R}\backslash\{0\}$, a vector $b\in \mathds{R}^{2n}$ and a matrix $S\in GL(2n,\mathds{R})$ with $S^*\omega=a\omega$ and $SB=\mu BS$, such that
\begin{align}\label{AutOsc}
\phiup(z,\xi,t)=\left(a z+\frac{1}{2}\omega(S\xi,\e^{t\mu B}b+b)+mt+\frac{1}{2}\omega(\e^{t\mu B}b,b),S\xi+\e^{t\mu B}b-b,\mu t\right).
\end{align}
\end{theorem}

\begin{proof}
One can check that a map satisfying condition (\ref{AutOsc}) is an automorphism.
So we only verify the other implication.
Let $\phiup$ be an automorphism.\\
First of all, we check how elements of $\{0\}\times\{0\}\times\mathds{R}$ will be mapped.
Therefore, suppose $\phiup(0,0,t)=(z(t),\xi(t),\mu(t))$.
Then 
\begin{align*}
\big(z(s+t),\xi(s+t),\mu(s+t)\big)&=\phiup(0,0,s)\phiup(0,0,t)\\
&=\big(z(s)+z(t)+\frac{1}{2}\omega(\xi(s),\e^{\mu(s)B}\xi(t)),\xi(s)+\e^{\mu(s)B}\xi(t),\mu(s)+\mu(t)\big).
\end{align*}
We notice that $\mu$ is a linear mapping, i.\ e. $\mu(t)=\mu t$ for some $\mu\in\mathbbm{R}\backslash\{0\}$.\\
Differentiating $\xi(s+t)=\xi(s)+\e^{\mu sB}\xi(t)$ with respect to $s$ and setting $s=0$, we get
\[\xi'(t)=\mu Bb+\mu B \xi(t),\]
where $\mu Bb=\xi'(0)$, $b\in\mathds{R}^{2n}$.
The solution of this ODE with $\xi(0)=0$ is \[\xi(t)=\e^{\mu t B}b-b.\]
Finally, the comparison of the first components shows that
\[z(s+t)+\frac{1}{2}\omega(b,\e^{(s+t)\mu B}b)=z(s)+\frac{1}{2}\omega(b,\e^{s\mu B}b)+z(t)+\frac{1}{2}\omega(b,\e^{t\mu B}b)\] and thus
\[z(t)=mt-\frac{1}{2}\omega(b,\e^{t\mu B}b)\] for some $m\in \mathds{R}$.

Since $H_n(\omega)$ is the commutator subgroup of $Osc_n(\omega,B)$, the automorphism $\phiup$ maps $H_n(\omega)$ onto itself.

Hence $\phiup(z,\xi,0)=(a z+\delta^T\xi,S\xi,0)$, where $S^*\omega=a\omega$ and $\delta\in\mathds{R}^{2n}$ (compare to \cite[p. 294]{To78}).

Thus $\phiup(z,\e^{tB}\xi,0)\phiup(0,0,t)=\phiup(0,0,t)\phiup(z,\xi,0)$ gives
\begin{align*}
&(\delta^T\e^{tB}\xi+a z+mt+\frac{1}{2}\omega(\e^{t\mu B}b,b)+\frac{1}{2}\omega(S\e^{tB}\xi,\e^{t\mu B}b-b),S\e^{tB}\xi+\e^{t\mu B}b-b,\mu t)\\
&=(mt+\frac{1}{2}\omega(\e^{t\mu B}b,b)+\delta^T\xi+a z+\frac{1}{2}\omega(\e^{t\mu B}b-b,\e^{\mu tB}S\xi),\e^{t\mu B}b-b+\e^{\mu t B}S\xi,\mu t).
\end{align*}
From the second component it follows, that $SB=\mu BS$. Since $\det(SB)=\mu^{2n}\det(BS)$, we get $\mu\in\left\{+1,-1\right\}$.

From the first component, we get $\delta^T\e^{tB}\xi+\omega(S\e^{tB}\xi,\e^{t\mu B}b-b)=\delta^T\xi$
for all $\xi\in\mathds{R}^{2n}$, $t\in\mathds{R}$. 
Differentiating and setting $t=0$ gives $\delta^TB\xi-\omega(SB\xi,b)=0$ for all $\xi$.\\
This completes the proof.
\end{proof}

So the automorphisms $\varphi$ satisfying $\varphi\big|_{H_n(\omega)}=id$ are of the form $(z,\xi,t)\mapsto(z+mt,\xi, t)$ with $m\in\mathds{R}$. This contradicts the assertion in \cite[p. 92]{MR85}.

\section{Lattices} \label{Sectionlattices}
Now, we study the lattices of oscillator groups and begin with an example.

\begin{example}\label{Bsp0}
Let $\Gamma_\omega$ denote the subgroup in $Osc_n(\omega,B)$ generated by $\big\{(1,0,0),(0,e_i,0)\mid i=1,\dots,2n\big\}$.
For each $z_0\in\mathds{R}$ and $\xi_0\in\mathds{R}^{2n}$ the subgroup $L:=\big\langle\Gamma_\omega\cup\{(z_0,\xi_0,1)\}\rangle$ is a lattice in $Osc_n(\omega,B)$ with $L\cap H_n(\omega)=\Gamma_\omega$ if and only if $\left(\omega(\xi_0,\e^Be_i),\e^Be_i,0\right)\in\Gamma_\omega$ for $i=1,\dots,2n$.

To see this, one can check that 
\begin{align}\label{LCapHeis}
(z,\xi,t)(v,\eta,0)(z,\xi,t)^{-1}&=(v+\omega(\xi,\e^{tB}\eta),\e^{tB}\eta,0)%\in L\cap H_n(\omega)
\end{align}
for each $(v,\eta,0)\in H_n(\omega)\cap L$ and $(z,\xi,t)\in L$.
\end{example}

If L=$\langle\Gamma_\omega\cup\{(0,\xi,1)\}\rangle$ defines a lattice in $Osc_n(\omega,B)$ with $L\cap H_n(\omega)=\Gamma_\omega$, then we just call it $L(\xi)$.
%We denote the lattice $L(\xi):=\langle(1,0,0),(0,e_i,0),(0,\xi,1)\mid i=1,\dots,2n\rangle$ in $Osc_1(\omega_r,B)$
Furthermore, we set $\Gamma_r:=\Gamma_{\omega_r}$.
Note that \[\Gamma_r=\left\{(z,\xi,0)\,\Big|\, \xi\in\mathds{Z}^{2n}, z\in\sum_{i=1}^n\frac{r_i}{2}\xi_{2i-1}\xi_i+\mathds{Z}\right\}.\]
In particular for $n=1$, we have $\Gamma_r=\mathds{Z}^3\times\{0\}$ for $r$ even and $\Gamma_r=\linebreak\left\{(z,\xi,0)\mid\xi\in\mathds{Z}^2, z\in\frac{1}{2}\xi_1\xi_2+\mathds{Z}\right\}$ for $r$ odd.
Moreover let $\Pi$ denote the projection on the last component and for a lattice $L$ of $Osc_n(\omega,B)$ we denote by $\Pi(L)$ the set $\Pi(L):=\big\{t\in\mathds{R}\mid \exists\, z,\xi: (z,\xi,t)\in L\big\}$.
Note that $\Pi(L)$ is a non-trivial discrete subgroup of $\mathds{R}$ for each lattice $L$ (compare to \cite[p. 90]{MR85}).

\begin{theorem} \label{theoremLCapHeis=GammaR}
Let $L$ be a lattice in $Osc_n(\omega,B)$.
Then there exists a uniquely determined $r\in\mathds{N}^{n}$ satisfying $r_i$ divides $r_{i+1}$ for $i=1,\dots,n-1$, a linear map $\tilde{B},$ an isomorphism \linebreak $\Phi:Osc_n(\omega,B)\rightarrow Osc_n(\omega_r,\tilde{B})$ and a $\xi_0\in\mathds{R}^{2n}$, such that \[\Phi(L)=L(\xi_0).\]
\end{theorem}

\begin{proof}
At first, note that $L\cap H_n(\omega)$ is a lattice in $H_n(\omega)$, see \cite[p. 50]{Ra72}.
Furthermore, we know from theorem 1.10 in \cite[p. 303]{To78}
%(see also theorem (2.4) in \cite[p. 255]{GW86})
 that there is a uniquely determined $r=(r_1,\dots,r_n)$ where $r_i\in \mathds{N}\backslash\{0\}$, $r_i\mid r_{i+1}$ for $i=1,\dots,n-1$ and an isomorphism $\varphi:H_n(\omega)\rightarrow H_n(\omega_r)$, $\varphi(z,\xi)=(\delta^T\xi+az,S\xi)$, where $S^*\omega_r=a\omega$, such that $\varphi(L\cap Heis\big(\mathds{R}^{2n},\omega)\big)=\Gamma_r$.
We choose $b$, such that $\delta^T\xi=\omega_r(S\xi,b)$ for all $\xi\in\mathds{R}^{2n}$.
Then
\[\overline{\varphi}(z,\xi,t)=\left(az+\frac{1}{2}\omega_r(S\xi,\e^{tSBS^{-1}}b+b)+\frac{1}{2}\omega_r(\e^{tSBS^{-1}}b,b),S\xi+\e^{tSBS^{-1}}b-b,t\right)\] is an isomorphism from $Osc_n(\omega,B)$ to $Osc_n(\omega_r,SBS^{-1})$, mapping $L\cap H_n(\omega)$ onto $\Gamma_r$.
Let $t_0$ denote the smallest positive element in $\Pi(\overline{\varphi}(L))$.
So there is a $z_0$ and a $\xi_0$, such that $(z_0,\xi_0,t_0)\in \overline{\varphi}(L)$.
The map \[\phi:(z,\xi,t)\mapsto\left(z-\frac{z_0}{t_0}t,\xi,\frac{t}{t_0}\right)\]
is an isomorphism  from $Osc_n(\omega_r,SBS^{-1})$ to $Osc_n(\omega_r,t_0SBS^{-1})$ such that \linebreak $\phi\big|_{H_n(\omega_r)}=id$ and $(z_0,\xi_0,t_0)$ maps to $(0,\xi_0,1)$.
Hence the theorem is proved.
\end{proof}

From now on we consider $n=1$.

\begin{lemma}\label{lemmaIsoOscMitGitter}
There is an isomorphism $\varphi:Osc_1(\omega_r,B)\rightarrow Osc_1(\omega_r,\tilde{B})$ mapping $\Gamma_r$ onto $\Gamma_r$ and satisfying $\Pi(\varphi(0,0,1))=\pm 1$ if and only if $B$ or $-B$ is conjugate to $\tilde{B}$ with respect to an integer matrix with determinant $\pm 1$.
\end{lemma}

\begin{proof}
At first, we will construct an isomorphism satisfying the conditions in the lemma. Let $B$ and $\tilde{B}$ be conjugate with respect to an integer matrix with determinant $\pm 1$.
This is sufficient to assume, since the map $\phi:(z,\xi,t)\mapsto(z,\xi,-t)$ is an isomorphism from $Osc_1(\omega_r,B)$ to $Osc_1(\omega_r,-B)$, satisfying the conditions in the lemma.
Let \[S=\left(\begin{matrix}s_1&s_2\\s_3&s_4\end{matrix}\right)\] be the integer conjugation matrix with determinant $\pm 1$, such that $\tilde{B}=SBS^{-1}$, $a=det(S)$ and
\[b:=\begin{cases}
        (0,0),&\text{ $s_1s_2$ and $s_3s_4$ are even}\\
        (0,\frac{1}{2}),&\text{ $s_1s_2$ is even and $s_3s_4$ is odd}\\
        (\frac{1}{2},0),&\text{ $s_3s_4$ is even and $s_1s_2$ is odd}\\
       \end{cases}.\]
We only get these three cases, since $\det(S)$ is odd.
Then \[\varphi:(z,\xi,t)\mapsto \big(az+\frac{1}{2}\omega_r(S\xi,\e^{t\tilde{B}}b+b)+\frac{1}{2}\omega_r(\e^{t\tilde{B}}b,b),S\xi+\e^{t\tilde{B}}b-b,t\big)\] is an isomorphism from $Osc_1(\omega_r,B)$ to $Osc_1(\omega_r,\tilde{B}),$ satisfying $\varphi(\Gamma_r)=\Gamma_r$ and $\Pi(\varphi(0,0,1))=\pm 1$.
Here we will verify this only for the case that $s_1s_2$ is even and $s_3s_4$ odd (the other cases run similar).
It's not hard to see that $\varphi$ is an isomorphism and $\Pi(\varphi(0,0,1))=1$.
Furthermore $\varphi(0,e_1,0)=\left(s_1\frac{r}{2},(s_1,s_3)^T,0\right)\in\Gamma_r$, since $s_1\frac{r}{2}\in\frac{r}{2}s_1s_3+\mathds{Z}$ and similarly
$\varphi(0,e_2,0)=\left(s_2\frac{r}{2},(s_2,s_4)^T,0\right)\in\Gamma_r$.
In addition
$(-\frac{r}{2},a(s_4,-s_3)^T,0)$ and $(0,a(-s_2,s_1)^T,0)$ are elements of $\Gamma_r$ and will be mapped to $(0,e_1,0)$ respectively $(0,e_2,0)$.
Since, moreover, $\varphi(1,0,0)=(a,0,0)=(\pm 1,0,0)$, we see finally that $\varphi(\Gamma_r)=\Gamma_r$.

So one direction of the lemma is verified.

Now, let $\varphi$ be an isomorphism satisfying the conditions in the lemma.
We know from Lemmas \ref{Iso00} and \ref{IsoO} that there is a $k\in\mathds{R}\backslash\{0\}$ and a $T\in SL(2,\mathds{R})$, such that $TBT^{-1}=k\tilde{B}$. 
So we can write $\varphi=\varphi_2\circ\varphi_1$, where \[\varphi_1(z,\xi,t)=(z,T\xi,kt)\] is an isomorphism from $Osc_1(\omega_r,B)$ to $Osc_1(\omega_r,\tilde{B})$ and $\varphi_2$ an automorphisms of $Osc_1(\omega_r,\tilde{B})$
Thus \[\varphi_2(z,\xi,t)=(a z+\frac{1}{2}\omega_r(\tilde{T}\xi,\e^{t\mu \tilde{B}}b+b)+mt+\frac{1}{2}\omega_r(\e^{t\mu \tilde{B}}b,b),\tilde{T}\xi+\e^{t\mu \tilde{B}}b-b,\mu t),\]
where $\det(\tilde{T})=a$, $\tilde{T}\tilde{B}\tilde{T}^{-1}=\mu\tilde{B}$ and $\mu\in\{+1,-1\}$.
Hence \[\varphi(z,\xi,t)=(a z+\frac{1}{2}\omega_r(\tilde{T}T\xi,\e^{kt\mu \tilde{B}}b+b)+mt+\frac{1}{2}\omega_r(\e^{kt\mu \tilde{B}}b,b),\tilde{T}T\xi+\e^{kt\mu \tilde{B}}b-b,\mu kt).\]

We get $\tilde{T}TBT^{-1}\tilde{T}^{-1}=\pm \tilde{B}$, since $TBT^{-1}=k\tilde{B}$, $\tilde{T}\tilde{B}\tilde{T}^{-1}=\mu \tilde{B}$ and $\Pi(\varphi(0,0,1))=\pm 1$.

In addition, $\varphi$ maps $(1,0,0)$ to $(\pm 1,0,0)$ and $(0,e_i,0)$ into $\Gamma_r$ for $i=1,2$. Hence $\det(\tilde{T}T)=\pm 1$ and $\tilde{T}Te_i\in\mathds{Z}^2$ for $i=1,2$.
Thus we obtain the assertion.
\end{proof}

To classify the lattices of oscillator groups, we have to choose representatives for the conjugacy classes of matrices, which appeared in the previous lemma. %Therefor we give this definition.
\begin{definition}
For $y\neq 0$ and $x\in\mathds{R}$ we denote  \[B_{x,y}:=\begin{pmatrix}\frac{x}{y}&-\frac{x^2}{y}-y\\\frac{1}{y}&-\frac{x}{y}\end{pmatrix}.\]
\end{definition}

The set of all $B_{x,y}$ is equal to the set of all matrices which are conjugate to $N_1$, and to the set of all $2\times 2$-matrices with determinant $1$ and trace $0$.
The subsets $\mathds{B}^+:=\big\{B_{x,y}\mid y>0\big\}$ and $\mathds{B}^-:=\big\{B_{x,y}\mid y<0\big\}$ are invariant under conjugation with elements of $SL(2,\mathds{R})$.
Furthermore, conjugation with elements of $GL(2,\mathds{R})$ with determinant $-1$ maps elements of $\mathds{B}^+$ to $\mathds{B}^-$ and reverse.

\begin{definition}
We define \[\mathds{F}_1:=\left\{(x,y)\in\mathds{R}^2\mid 0\leq x\leq\frac{1}{2}, y>0, x^2+y^2\geq 1\right\}\] and \[\mathds{F}=\mathds{F}_1\cup \left\{(x,y)\in\mathds{R}^{2}\mid -\frac{1}{2}<x<0, y>0, x^2+y^2>1\right\}.\]
\end{definition}

Note that the map $\iota: B_{x,y}\mapsto x+iy$ is a bijection from $\big\{B_{x,y}\mid y>0\big\}$ to the upper half plane of $\mathds{C}$, satisfying $\iota(AB_{x,y}A^{-1})=A(x+iy)$ for all $A=\left(\begin{smallmatrix}a&b\\c&d\end{smallmatrix}\right)\in SL(2,\mathds{R})$, where $Az=\frac{az+b}{cz+d}$.
With this in mind, we rewrite the theorem in \cite[p. 109]{KK98}:
\begin{remark} \label{Fundamentalbereich} 
For all $B_{x',y'}$, $y'>0$ there is a uniquely defined $(x,y)\in\mathds{F}$ and an $S\in SL(2,\mathds{Z})$, such that $SB_{x',y'}S^{-1}=B_{x,y}$.
\end{remark}

\begin{theorem} \label{Klassifik2}
Let $L$ be a lattice of $Osc_1(\omega,B)$.
Then there exist 
\begin{itemize}
    \item a uniquely determined $r\in\mathds{N}\backslash\{0\}$,
    \item a uniquely determined $\lambda=\lambda_0+k\pi$ with $k\in\mathds{N}$ and 
$\lambda_0\in\left\{\frac{1}{3}\pi,\frac{1}{2}\pi,\frac{2}{3}\pi,\pi\right\}$
    \item and a uniquely determined \[(x,y)\begin{cases}
=(\frac{1}{2},\frac{\sqrt{3}}{2}),&\lambda_0\in\{\frac{1}{3}\pi,\frac{2}{3}\pi\}\\
=(0,1),&\lambda_0=\frac{1}{2}\pi\\
\in\mathds{F}_1,&\lambda_0=\pi
\end{cases},\]
    \end{itemize}
and an isomorphism $\varphi:Osc_1(\omega,B)\rightarrow Osc_1(\omega_r,\lambda B_{x,y})$ satisfying $\varphi(L)\cap H_1(\omega_r)=\Gamma_r$ and $\Pi(\varphi(L))=\mathds{Z}$.\\

Conversely, for any such data $(r,\lambda,(x,y))$ there exists a lattice $L$ in $Osc_1(\omega_r,\lambda B_{x,y})$ satisfying $L\cap H_1(\omega_r)=\Gamma_r$ and $\Pi(L)=\mathds{Z}$.
\end{theorem}

\begin{proof}%[Proof of Theorem \ref{Klassifik2}]
 Because of Theorem \ref{theoremLCapHeis=GammaR} we can suppose that $\omega=\omega_r$ and the lattice $L$ given in $Osc_1(\omega_r,B)$ satisfies $L\cap H_1(\omega_r)=\Gamma_r$ and $\Pi(L)=\mathds{Z}$.
The procedure is to find a $\tilde{B}(=\lambda B_{x,y})$ for a given $B$, such that there is an isomorphism $\varphi$ from $Osc_1(\omega_r,B)$ to $Osc_1(\omega_r,\tilde{B})$, mapping $\Gamma_r$ onto $\Gamma_r$ and satisfying $\Pi(\varphi(L))=\mathds{Z}$.

Let $\lambda\in\mathds{R}$ be the positive imaginary part of the eigenvalue of $B$. Then we see that $B$ and $N_\lambda$ are conjugate. Hence $\e^B$ and $\e^{N_\lambda}=\left(\begin{smallmatrix}\cos\lambda&-\sin\lambda\\\sin\lambda&\cos\lambda\end{smallmatrix}\right)$ are conjugate. So $\tr(\e^B)=2\cos\lambda$ and $\cos\lambda\in\left\{-1,-\frac{1}{2},0,\frac{1}{2},1\right\}$, since $\e^B\in SL(2,\mathds{Z})$. Thus $\lambda=\lambda_0+k\pi\neq 0$, where 
$\lambda_0\in\left\{\frac{\pi}{3},\frac{\pi}{2},\frac{2\pi}{3},\pi\right\}$
and $k\in\mathds{Z}$.
There is an $x'\in\mathds{R}$ and a $y'\neq 0$, such that $B=\lambda B_{x',y'}$. Now we can use that $B_{x',-y'}=-B_{x',y'}$ to assume that $y'>0$.
Now we choose an $S\in SL(2,\mathds{Z})$, such that $SB_{x',y'}S^{-1}=B_{x,y}$, where $x$ and $y$ satisfy the conditions in Remark \ref{Fundamentalbereich}.
%So $SBS^{-1}=\lambda B_{x,y}$, where $x\in(\frac{-1}{2},\frac{1}{2}]$, $y>0$ and $x^2+y^2\geq 1$, respectively $x^2+y^2>1$, for $x<0$.
So we set $\tilde{B}:=|\lambda| B_{|x|,y}$.
Because of Lemma \ref{lemmaIsoOscMitGitter} and the fact that 
$\left(\begin{smallmatrix}1&0\\0&-1\end{smallmatrix}\right)B_{x,y}\left(\begin{smallmatrix}1&0\\0&-1\end{smallmatrix}\right)=-B_{-x,y}$,
there is an isomorphism $Osc_1(\omega_r,B)\rightarrow Osc_1(\omega_r,\tilde{B})$ mapping $\Gamma_r$ onto $\Gamma_r$ and satisfying $\Pi(\varphi(L))=\mathds{Z}$.

Now we want to see that $\lambda>0$ and $(x,y)\in\mathds{F}_1$ are uniquely determined.
Let $\lambda,\lambda_2>0$ and $(x,y),(x_2,y_2)\in\mathds{F}_1$.
Suppose that there is an isomorphism $\varphi:Osc_1(\omega_r,\lambda B_{x,y})\rightarrow Osc_1(\omega_r,\lambda_2B_{x_2,y_2})$ mapping $\Gamma_r$ onto $\Gamma_r$ and satisfying $\Pi(\varphi(0,0,1))=\pm 1$. Then we will see that $\lambda B_{x,y}=\lambda_2 B_{x_2,y_2}$.
At first Lemma \ref{lemmaIsoOscMitGitter} gives that $\pm\lambda B_{x,y}$ and $\lambda_2 B_{x_2,y_2}$ are conjugate. Since $\det(\pm\lambda B_{x,y})=\det(\lambda_2B_{x_2,y_2})$, we get $\lambda=\lambda_2$.
Hence $\pm B_{x,y}$ and $B_{x_2,y_2}$ are conjugate.

For a better readability we write $\mathds{Z}_1$-conjugate, if two matrices are conjugate with respect to a matrix in $SL(2,\mathds{Z})$ and $\mathds{Z}_{-1}$-conjugate, if two matrices are conjugate with respect to an integer matrix, having determinant $-1$.

%Note that for all $A\in SL(2,\mathds{R})$ we get $AB_{x,y}A^{-1}=B_{x',y'}$ for some $x'\in\mathds{R}$ and $y'\neq 0$ satisfying $yy'>0$.
It is clear that $-B_{x,y}=B_{x,-y}$ and $B_{x_2,y_2}$ are not $\mathds{Z}_1$-conjugate. In addition $B_{x,y}$ and $B_{x_2,y_2}$ are not $\mathds{Z}_{-1}$-conjugate.
If $B_{x,y}$ and $B_{x_2,y_2}$ are $\mathds{Z}_1$-conjugate, then $(x,y)=(x_2,y_2)$ because of Remark \ref{Fundamentalbereich}.
Finally, if $-B_{x,y}$ and $B_{x_2,y_2}$ are $\mathds{Z}_{-1}$-conjugate, then $B_{-x,y}$ and $B_{x_2,y_2}$ are $\mathds{Z}_1$-conjugate. Using Remark \ref{Fundamentalbereich} gives $x=x_2\in\{0,\frac{1}{2}\}$ and $y=y_2$ easily follows.

At last, we want to see how $\lambda$ and $(x,y)$ fit together.
Suppose $B=\lambda B_{x,y}$, where $(x,y)\in\mathds{F}_1$ and $\lambda=\lambda_0+k\pi$, satisfying $k\in\mathds{N}$ and 
$\lambda_0\in\left\{\frac{1}{3}\pi,\frac{1}{2}\pi,\frac{2}{3}\pi\right\}$.
Let $L$ be a lattice in $Osc_1(\omega_r,\lambda B_{x,y})$ such that $L\cap H_1(\omega_r)=\Gamma_r$ and $\Pi(L)=\mathds{Z}$.
Then, using equation (\ref{LCapHeis}) with $t=1$,
\begin{align}\label{eHochB}
&\e^{B}=\begin{pmatrix}\cos\lambda+\frac{x}{y}\sin\lambda&-\frac{x^2}{y}\sin\lambda-y\sin\lambda\\\frac{1}{y}\sin\lambda&-\frac{x}{y}\sin\lambda+\cos\lambda\end{pmatrix}\in SL(2,\mathds{Z}).
\end{align}
Thus $\frac{1}{y}\sin\lambda\in\mathds{Z}$ and $y=|\sin\lambda|$.
So $x=\frac{1}{2}$ if $\cos\lambda=\pm\frac{1}{2}$, and $x=0$ if $\cos\lambda=0$.
We get
\[B=\begin{cases}
        \lambda B_{0,1},&\cos\lambda=0\\
        \lambda B_{\frac{1}{2},\frac{\sqrt{3}}{2}},&\cos\lambda=\pm\frac{1}{2}\\
       \end{cases}.\]
Conversely, it is obvious that $\e^{\lambda B_{x,y}}\in SL(2,\mathds{Z})$ for the data $(r,\lambda,(x,y))$ described in the theorem.
Hence there is a lattice satisfying the claimed conditions (see Example \ref{Bsp0}, where $\xi_0=(0,0)$, $\xi_0=(0,\frac{1}{2})$ or $\xi_0=(\frac{1}{2},0)$ depending on $r$ and $\e^B$).
\end{proof}

%\label{AbbildungseigenschaftAut}
\begin{lemma}\label{Autos} \label{AbbildungseigenschaftAut} %(Automorphismen, die das Gitter in der Heisenberggruppe auf sich abbilden)\\
Let $B=\lambda B_{x,y}$, where $x\in\mathds{R}$ and $y>0$. Then an automorphism $\varphi$ of $Osc_1(\omega_r,B)$  maps $\Gamma_r$ onto itself if and only if
\begin{align}\label{AutNummer}
\varphi(z,\xi,t)=\big(\mu z+\frac{1}{2}\omega_r(S\xi,\e^{t\mu B}b+b)+mt+\frac{1}{2}\omega_r(\e^{t\mu B}b,b),S\xi+\e^{t\mu B}b-b,\mu t\big),
\end{align}
where $m\in\mathds{R}$, $\mu\in\{\pm 1\}$, S is an integer matrix
\[S=\begin{pmatrix}s_1&s_2\\s_3&s_4\end{pmatrix}\in\left\{ \e^{tB},\e^{tB}\begin{pmatrix}1&-2x\\0&-1\end{pmatrix}\mid t\in\mathds{R}\right\}\]
with $\det(S)=\mu$ and $b\in(\frac{\mathds{Z}}{r},\frac{\mathds{Z}}{r}):=\frac{1}{r}\mathds{Z}\times\frac{1}{r}\mathds{Z}$, if $r$ is even, respectively
\[b=(b_1,b_2)\in\begin{cases}
        (\frac{\mathds{Z}}{r},\frac{\mathds{Z}}{r}),&s_1s_2\text{ and }s_3s_4\text{ are even}\\
        (\frac{\mathds{Z}}{r},\frac{1}{2r}+\frac{\mathds{Z}}{r}),&s_1s_2\text{ is even and }s_3s_4\text{ is odd}\\
        (\frac{1}{2r}+\frac{\mathds{Z}}{r},\frac{\mathds{Z}}{r}),&s_3s_4\text{is even and }s_1s_2\text{ is odd}\\
     \end{cases},\]
if $r$ is odd.
For short, we call such an automorphism $\Gamma_r$-preserving.
\end{lemma}

\begin{proof}
It's not hard to show that an automorphism as defined in the theorem maps $\Gamma_r$ onto $\Gamma_r$.
%Sei $B=\lambda B_{x,y}$ gegeben. Dann ist für $T=\begin{pmatrix}\sqrt{y}&\frac{x}{\sqrt{y}}\\&\frac{1}{\sqrt{y}}\end{pmatrix}$ $B=TN_\lambda T^{-1}$,
Let $\varphi$ be an automorphism in $Osc_1(\omega_r,B)$, given as in (\ref{AutOsc}) , mapping $\Gamma_r$ onto itself.
Then $a=\pm 1$. In addition $\det S=a$, since $\det(S)\omega_r=S^*\omega_r=a\omega_r$.
Since $B=\lambda B_{x,y}$ and $SB_{x,y}S^{-1}=\mu B_{x,y}=B_{x,\mu y}$, we get $\det(S)=\mu$.

Because $\varphi(0,e_i,0)=(\omega_r(Se_i,b),Se_i,0)$, the invertible matrix $S$ has integer entries.

For \[T=\begin{pmatrix}\sqrt{y}&\frac{x}{\sqrt{y}}\\0&\frac{1}{\sqrt{y}}\end{pmatrix},\] it holds that $TN_1 T^{-1}=B_{x,y}$.
Since $SB=\mu BS$ we see that $STN_1T^{-1}=\mu TN_1T^{-1}ST$.
Then $T^{-1}ST$ and $N_1$ (anti-)commute. Hence $T^{-1}ST\in O(2,\mathds{R})$.
Thus $T^{-1}ST=\e^{t\lambda N_1}\left(\begin{smallmatrix}1&0\\0&\mu\end{smallmatrix}\right)$ and $S=\e^{tB}T\left(\begin{smallmatrix}1&0\\0&\mu\end{smallmatrix}\right)T^{-1}$, for some $t\in\mathds{R}$.
If $\mu=1$, then $S=\e^{tB}$, if $\mu=-1$, then \[S=\e^{tB}\begin{pmatrix}1&-2x\\0&-1\end{pmatrix}.\]
Hence $S$ is an integer matrix in \[\left\{\e^{tB},\e^{tB}\begin{pmatrix}1&-2x\\0&-1\end{pmatrix}\mid t\in\mathds{R}\right\}.\]

%Note that, if $TN_\lambda T^{-1}=B$ and $SB=\mu BS$ then $T^{-1}ST\in O(2,\mathds R)$.
%Hence $T^{-1}ST\left(\begin{smallmatrix}1&0\\0&\mu\end{smallmatrix}\right)\in SO(2,\mathds{R})$

Finally we want to see how $b$ looks like.
Let $r$ be even.
Then $\Gamma_r=\mathds{Z}^3$ and hence \[\omega_r\big((s_1,s_3),(b_1,b_2)\big),\omega_r\big((s_2,s_4),(b_1,b_2)\big)\in\mathds{Z}.\]
Thus $b\in(\frac{1}{r}\mathds{Z},\frac{1}{r}\mathds{Z})$.

Now let $r$ be odd, so $\Gamma_r=\left\{(z,\left(\xi_1,\xi_2),0\right)\mid z\in\frac{1}{2}\xi_1\xi_2+\mathds{Z}\right\}$.
%Since $\omega_r(\e^Be_i,b)\in\frac{e_1^T\e^Be_ie_2^T\e^Be_i}{2}+\mathds{Z}$,

We see that there are the same three cases to consider as in the previous lemma. Here, we only check the case that $s_1s_2$ is even and $s_3s_4$ is odd, especially $s_1$ is even (the case that $s_2$ is even runs similarly).
Since $\det(S)$ is odd, we get that $s_2$ is odd.
Then $\omega_r\big((s_1,s_3),(b_1,b_2)\big)\in\mathds{Z}$ and $\omega_r\big((s_2,s_4),(b_1,b_2))\in\mathds{Z}+\frac{1}{2}.$
Hence \[r\begin{pmatrix}b_1\\b_2\end{pmatrix}=\det(S)\begin{pmatrix}-s_2&s_1\\-s_4&s_3\end{pmatrix}\begin{pmatrix}k_1\\k_2+\frac{1}{2}\end{pmatrix},\]
for some $k_1,k_2\in\mathds{Z}$.
Thus $rb_1\in\mathds{Z}$ and $rb_2\in\mathds{Z}+\frac{1}{2}$.

Finally, for each case we obtain the assertion.
\end{proof}

Now we can completely classify the lattices of $Osc_1(\omega,B)$ by using Theorem \ref{Klassifik2} and the following one.

\begin{theorem}\label{Klassifik}
Suppose $B=\lambda B_{x,y}$, $(x,y)\in\mathds{F}_1$ and $\lambda=\lambda_0+k\pi$, where $k\in\mathds{N}\cup\{0\}$ and 
$\lambda_0\in\left\{\frac{1}{3}\pi,\frac{1}{2}\pi,\frac{2}{3}\pi,\pi\right\}$.
Let $L=\big\langle\Gamma_r \cup \{(z,\xi,1)\}\big\rangle$ be a lattice in $Osc_1(\omega_r,B)$ with $L\cap H_1(\omega_r)=\Gamma_r$.
Then there is a uniquely defined $\xi_0$ to extract from the list in Appendix \ref{AppA} and an automorphism $\varphi$ of $Osc_1(\omega_r,B)$, such that $\varphi(L)=L(\xi_0)$.
\end{theorem}

\begin{proof}
%\paragraph{first part}
We will say that $\xi$ and $\eta$ are equivalent if there is an automorphism mapping $L(\xi)$ onto $L(\eta)$.

\paragraph{Existence}
We begin with showing the existence of a $\xi_0$ in the list in Appendix \ref{AppA} and an $\Gamma_r$-preserving automorphism $\varphi$ for each lattice $L$ satisfying that $L\cap H_1(\omega_r)=\Gamma_r$, such that $\varphi(L)=L(\xi_0)$.

The proof will be divided into parts, dependent on the value of $\lambda$.
We, always, use $\e^B$, which we can compute with equation (\ref{eHochB}) in the proof of Theorem \ref{Klassifik2}. 

First we notice, depending on $r$, which values are possible for $\xi\in\mathds{R}^{2}$, such that $\langle\Gamma_r\cup\{(z,\xi,1)\}\rangle$ defines a lattice.
Therefor we use example \ref{Bsp0}. %equation (\ref{LCapHeis})
Afterward, we give automorphisms given as in (\ref{AutNummer}) in Lemma \ref{Autos}, which map $\big\langle\Gamma_r\cup\{(z,\xi,1)\}\big\rangle$ onto $L(\xi_0)$ for some $\xi_0$ from the list. 

It is clear that the automorphism $(z,\xi,t)\mapsto(z+mt,\xi,t)$ for some $m$ maps $(z_0,\xi_0,1)$ to $(0,\xi_0,1)$. Furthermore, we know from $S$ how the last component will be mapped by the automorphism (\ref{AutNummer}), since $\mu=\det(S)$.
%At that, we can neglect the first component, because of the automorphism $(z,\xi,t)\mapsto(z+mt,\xi,t)$, and the last component, since $\mu=\det(S)$.
Therefore we just give $S$ and $b$ and check, how $(0,\xi,1)$ will be mapped.
Each automorphism we give is $\Gamma_r$-preserving, so we don't make mention of that always. For further argumentation let $\xi=(\xi_1,\xi_2)$ denote an arbitrary vector in $\mathds{R}^2$ which satisfies the condition in Example \ref{Bsp0}.
Now, we begin with showing the existence of a vector of the list in Appendix \ref{AppA}.\\
\\
\noindent\hspace*{20mm}%
Suppose $\lambda=\lambda_0+2k\pi$, $k\in\mathds{N}$ and $\lambda_0\in\{\frac{\pi}{3},\frac{5\pi}{3}\}$. Let $r$ be even.
Then $\xi\in(\frac{1}{r}\mathds{Z},\frac{1}{r}\mathds{Z})$. In addition $\det(\e^B-E_2)=1$. So  the automorphism (\ref{AutNummer}), where $b:=(\e^B-E_2)^{-1}\xi\in(\frac{1}{r}\mathds{Z},\frac{1}{r}\mathds{Z})$ and $S:=E_2$, maps $(0,0,1)$ to $(0,\xi,1)$. Thus $(0,0)$ and $\xi$ are equivalent.
Now let $r$ be odd.
For $\lambda_0=\frac{\pi}{3}$ one can check that $\xi\in(\frac{1}{r}\mathds{Z}+\frac{1}{2r},\frac{1}{r}\mathds{Z})$. Since, again, $\det(\e^B-E_2)=1$, the automorphism (\ref{AutNummer}) given by $b:=(\e^B-E_2)^{-1}(\xi_1-\frac{1}{2r},\xi_2)\in(\frac{1}{r}\mathds{Z},\frac{1}{r}\mathds{Z})$ and $S:=E_2$ shows that $(\frac{1}{2r},0)$ and $\xi$ are equivalent.
For $\lambda_0=\frac{5}{3}\pi$ an analogous argumentation holds, except having $\frac{1}{2r}$ in the other component.\\
\\
\noindent\hspace*{20mm}%
Suppose $\lambda=\lambda_0+2k\pi$, $k\in\mathds{N}$ and $\lambda_0\in\{\frac{\pi}{2},\frac{3\pi}{2}\}$. Then $\xi\in(\frac{1}{r}\mathds{Z},\frac{1}{r}\mathds{Z})$. In addition $\det(\e^B-E_2)=2$.
If $r\xi_1-r\xi_2$ is even, then the automorphism defined by  $b:=(\e^B-E_2)^{-1}\xi\in(\frac{1}{r}\mathds{Z},\frac{1}{r}\mathds{Z})$ and $S:=E_2$ maps $(0,0,1)$ to $(0,\xi,1)$. Hence $(0,0)$ and $\xi$ are equivalent.
If $r\xi_1-r\xi_2$ is odd, then we set $S:=E_2$ and $b:=(\e^B-E_2)^{-1}(\xi_1,\xi_2-\frac{1}{r})\in(\frac{1}{r}\mathds{Z},\frac{1}{r}\mathds{Z})$. Thus $(0,\frac{1}{r})$ is equivalent to $\xi$.
If additionally $r$ is odd, we instead set $b:=(e^B-E_2)^{-1}(1+\xi_1,\xi_2)\in(\frac{1}{r}\mathds{Z},\frac{1}{r}\mathds{Z})$. The related automorphism shows that $(0,0)$ is equivalent to $(1,0)+\xi$, which is equivalent to $\xi$.\\
\\
\noindent\hspace*{20mm}%
Suppose $\lambda=\lambda_0+2k\pi$, $k\in\mathds{N}$ and $\lambda_0\in\{\frac{2\pi}{3},\frac{4\pi}{3}\}$.
Let $r$ be even. Then $\xi\in(\frac{1}{r}\mathds{Z},\frac{1}{r}\mathds{Z})$.
If $r\xi_1+r\xi_2\equiv 0(3)$, the automorphism given by $b:=(\e^B-E_2)^{-1}\xi\in(\frac{1}{r}\mathds{Z},\frac{1}{r}\mathds{Z})$ and $S:=E_2$ maps $(0,0,1)$ to $(0,\xi,1)$. Hence $(0,0)$ and $\xi$ are equivalent.
If $r\xi_1+r\xi_2\equiv 1(3)$, then the automorphism defined by $b:=(\e^B-E_2)^{-1}(\xi_1-\frac{1}{r},\xi_2)\in(\frac{1}{r}\mathds{Z},\frac{1}{r}\mathds{Z})$ and $S:=E_2$ shows that $(\frac{1}{r},0)$ and $\xi$ are equivalent.
If $r\xi_1+r\xi_2\equiv 2(3)$, we set $b:=(\e^B-E_2)^{-1}(\xi_1-\frac{1}{r},\xi_2-\frac{1}{r})\in(\frac{1}{r}\mathds{Z},\frac{1}{r}\mathds{Z})$ and $S:=\left(\begin{smallmatrix}1&-1\\1&0\end{smallmatrix}\right)$. Thus $(\frac{1}{r},0)$ and $\xi$ are equivalent.
If additionally $3$ is not a factor of $r$, we set $b:=(\e^B-E_2)^{-1}(x+\xi_1,\xi_2)\in(\frac{1}{r}\mathds{Z},\frac{1}{r}\mathds{Z})$, where $rx+r\xi_1+r\xi_2\equiv 0(3)$, and $S:=E_2$. So we get that $(0,0)$ is equivalent to $(x,0)+\xi$, which is equivalent to $\xi$.

Now let $r$ be odd.
We obtain $\lambda_0=\frac{2}{3}\pi$. Using equation (\ref{LCapHeis}), we get $\xi\in(\frac{1}{r}\mathds{Z},\frac{1}{r}\mathds{Z}+\frac{1}{2r})$.
If $r\xi_1+r\xi_2-1/2\equiv 0(3)$, then the automorphism given by $b:=(\e^B-E_2)^{-1}(\xi_1,\xi_2-\frac{1}{2r})\in(\frac{1}{r}\mathds{Z},\frac{1}{r}\mathds{Z})$ and $S:=E_2$ shows that $(0,\frac{1}{2r})$ and $\xi$ are equivalent.
If $r\xi_1+r\xi_2-1/2\equiv 1(3)$, then we set $b:=(\e^B-E_2)^{-1}(\xi_1-\frac{1}{r},\xi_2-\frac{1}{2r})\in(\frac{1}{r}\mathds{Z},\frac{1}{r}\mathds{Z})$ and $S:=E_2$. The related automorphism shows that $(\frac{1}{r},\frac{1}{2r})$  and $\xi$ are equivalent.
If $r\xi_1+r\xi_2-1/2\equiv 2(3)$, then we set $b:=(\e^B-E_2)^{-1}(\xi_1,\xi_2+\frac{1}{2r})\in(\frac{1}{r}\mathds{Z},\frac{1}{r}\mathds{Z})$,  $S:=-E_2$ and get that $(0,\frac{1}{2r})$ and $\xi$ are equivalent.

If additionally $3$ is not a factor of $r$, then we set $b:=(\e^B-E_2)^{-1}(x+\xi_1,\frac{1}{2r}+\xi_2)\in(\frac{1}{r}\mathds{Z},\frac{1}{r}\mathds{Z})$, where $rx+r\xi_1+r\xi_2-1/2\equiv 0(3)$, and $S:=E_2$. The related automorphism shows that $(0,\frac{1}{2r})$ is equivalent to $(x,0)+\xi$, which is equivalent to $\xi$.

For $\lambda_0=\frac{4}{3}\pi$ there is an analogue argumentation, where $\frac{1}{2r}$ is in the other component.\\
\\
\noindent\hspace*{20mm}%
Suppose $\lambda=\pi+2\pi k$, $k\in\mathds{N}$. Then $\xi_0\in(\frac{1}{r}\mathds{Z},\frac{1}{r}\mathds{Z})$.
Let $r$ be odd. So there are $x',y'\in\{0,1\}$, such that $b:=(\e^B-E_2)^{-1}(\xi_1+x',\xi_2+y')\in(\frac{1}{r}\mathds{Z},\frac{1}{r}\mathds{Z})$ and the automorphism defined by $S=E_2$ and the above given $b$ shows that $(0,0)$ is equivalent to $\xi+(x',y')$, which is equivalent to $\xi$.
Now let $r$ be even.
We set $\eta:=\frac{1}{2r}(-1+(-1)^{r\xi_1},-1+(-1)^{r\xi_2})$. Then $\eta\in\{(0,0),(\frac{1}{r},0),(0,\frac{1}{r}),(\frac{1}{r},\frac{1}{r})\}$ and the automorphism given by $S=E_2$ and $b:=(\e^B-E_2)^{-1}(\xi_0+\eta)\in(\frac{1}{r}\mathds{Z},\frac{1}{r}\mathds{Z})$ maps $(0,\eta,1)$ to $(0,\xi,1)$.

If $x^2+y^2=1$, then the automorphism given by $S=\left(\begin{smallmatrix}0&1\\1&0\end{smallmatrix}\right)$ and $b=(0,0)$ %$(z,\xi,t)\mapsto(-z,\left(\begin{smallmatrix}0&1\\1&0\end{smallmatrix}\right)\xi,-t)$
maps $\big(0,(\frac{1}{r},0),1\big)$ to $\big(0,(0,\frac{1}{r}),-1\big)=\big(0,(0,\frac{1}{r}),1\big)^{-1}$. Thus $(\frac{1}{r},0)$ and $(0,\frac{1}{r})$ are equivalent.\\
If $x=\frac{1}{2}$, then the automorphism given by $S=\left(\begin{smallmatrix}1&-1\\0&-1\end{smallmatrix}\right)$ and 
$b=(-\frac{1}{r},-\frac{1}{r})$ maps $\left(0,(0,\frac{1}{r}),1\right)$ to $\big(0,(\frac{1}{r},\frac{1}{r}),-1\big)=\big(0,(\frac{1}{r},\frac{1}{r}),1\big)^{-1}$. Hence $(0,\frac{1}{r})$ and $(\frac{1}{r},\frac{1}{r})$ are equivalent.
If $(x,y)=(\frac{1}{2},\frac{\sqrt{3}}{2})$, then both automorphisms can be used. Thus every $\xi_0\in(\frac{1}{r}\mathds{Z},\frac{1}{r}\mathds{Z})$ is equivalent to $(0,0)$ or $(\frac{1}{r},\frac{1}{r})$.\\
\\
\noindent\hspace*{20mm}%
%\section{Ordnung 1 B beliebig}
We come to the last case: Suppose $\lambda=2\pi k$, $k\in\mathds{N}$.\\
Using equation (\ref{LCapHeis}) we get $(\xi_1,\xi_2)\in(\frac{1}{r}\mathds{Z},\frac{1}{r}\mathds{Z})$.

In this case we can neglect $b$, since $S\xi+\e^{B}b-b=S\xi+b-b=S\xi$.
Instead, it's more important to consider all the finitely many integer matrices in \[\left\{\e^{tB_{x,y}},\e^{tB_{x,y}}\begin{pmatrix}1&-2x\\0&-1\end{pmatrix}\mid t\in\mathds{R}\right\},\] 
for $B_{x,y}$.

At first we give some automorphisms which map $\Gamma_r$ onto itself for all $x$ and $y$. Afterwards, we restrict our observation to the different cases.

For each $\big(0,(\xi_1,\xi_2),1\big)$ there are $0\leq k, l<r$ such that $\Gamma_r\cup\big\{\big(0,(\xi_1,\xi_2),1\big)\big\}$ and $\Gamma_r\cup\left\{\big(0,(\frac{k}{r},\frac{l}{r}),1\big)\right\}$ generate the same lattice.

The automorphism where $S=-E_2$ maps $(0,(\frac{k}{r},\frac{l}{r}),1)$ to $(0,(-1,-1),0)(0,(\frac{r-k}{r},\frac{r-l}{r}),1)$.

Thus there is for each $\xi\in(\frac{1}{r}\mathds{Z},\frac{1}{r}\mathds{Z})$ an equivalent \[\xi_0\in M_1:=\left\{\big(\frac{k}{r},\frac{l}{r}\big)\mid 0\leq k,l\leq \frac{r}{2}\right\}\cup \left\{\big(\frac{k}{r},\frac{l}{r}\big)\mid 0<k<\frac{r}{2}<l<r\right\}.\]
Thus the first row is verified.

If additionally $x=0$, then the automorphism where $S=\left(\begin{smallmatrix}1&0\\0&-1\end{smallmatrix}\right)$ maps $\big(0,(\frac{k}{r},\frac{l}{r}),1\big)$ to a $\big((0,(0,-1),0)(0,(\frac{k}{r},\frac{r-l}{r}),1)\big)^{-1}$.
Thus the second row is verified.

If $x=\frac{1}{2}$ instead, then the automorphism where $S=\left(\begin{smallmatrix}-1&1\\0&1\end{smallmatrix}\right)$ maps $\big(0,(\frac{k}{r},\frac{l}{r}),1\big)$, where $k>l$, to $\big(0,(\frac{k-l}{r},\frac{r-l}{r}),1\big)^{-1}$.
So we can narrow the set of all $\xi_0$ down to \[\left\{\big(\frac{k}{r},\frac{l}{r}\big)\mid 0\leq k\leq l\leq \frac{r}{2}\right\}\cup\left\{\big(\frac{k}{r},\frac{l}{r}\big)\mid 0<k<\frac{r}{2}<l<r\right\}\cup\left\{\big(\frac{k}{r},0\big)\mid 0\leq k\leq\frac{r}{2}\right\}.\]

Furthermore $S=-\left(\begin{smallmatrix}-1&1\\0&1\end{smallmatrix}\right)$ gives an $\Gamma_r$-preserving automorphism, which maps $\big(0,(\frac{k}{r},\frac{l}{r}),1\big)$, where $l\geq k$, to $\big(0,(\frac{l-k}{r},\frac{l}{r}),1\big)^{-1}$.
Hence the third row follows.

For further argumentation suppose $x^2+y^2=1$.
We know that there is for each $\xi\in(\frac{1}{r}\mathds{Z},\frac{1}{r}\mathds{Z})$ an equivalent $\xi_0\in M_1$.
We can still restrict this set.
The automorphism where $S=\left(\begin{smallmatrix}0&-1\\-1&0\end{smallmatrix}\right)$ maps $\big(0,(\frac{k}{r},\frac{l}{r}),1\big)$ to $\big(0,(-\frac{l}{r},-\frac{k}{r}),-1\big)=\big(0,(\frac{l}{r},\frac{k}{r}),1\big)^{-1}$.
So there is for each $\xi\in(\frac{1}{r}\mathds{Z},\frac{1}{r}\mathds{Z})$ an equivalent \[\xi_0\in\left\{\big(\frac{k}{r},\frac{l}{r}\big)\mid 0\leq k\leq l\leq \frac{r}{2}\right\}\cup \left\{\big(\frac{k}{r},\frac{l}{r}\big)\mid 0<k<\frac{r}{2}<l<r\right\}.\]
Additionally the automorphism where $S=\left(\begin{smallmatrix}0&1\\1&0\end{smallmatrix}\right)$ maps $\big(0,(\frac{k}{r},\frac{l}{r}),1\big)$ to
\[\left(\big(0,(-1,0),0\big)\big(0,(0,-1),0\big)\big(0,(\frac{r-l}{r},\frac{r-k}{r}),1\big)\right)^{-1}.\]

Hence for $x^2+y^2=1$ and $\xi\in(\frac{1}{r}\mathds{Z},\frac{1}{r}\mathds{Z})$ there is an equivalent 
\[\xi_0\in M_2:=\left\{\big(\frac{k}{r},\frac{l}{r}\big)\mid 0\leq k\leq l\leq \frac{r}{2}\right\}\cup \left\{\big(\frac{k}{r},\frac{l}{r}\big)\mid k+l\leq r, 0<k<\frac{r}{2}<l<r\right\}.\]

Now we want to see, which elements in $M_2$ are equivalent, if $(x,y)=(0,1)$ or  $(x,y)=(\frac{1}{2},\frac{\sqrt{3}}{2})$.

So suppose $(x,y)=(0,1)$.
The automorphism where $S=\left(\begin{smallmatrix}-1&0\\0&1\end{smallmatrix}\right)$ maps $\big(0,(\frac{k}{r},\frac{l}{r}),1\big)$ to 
$\big(\big(0,(\frac{k}{r},\frac{r-l}{r}),1\big)\big(0,(0,-1),0\big)\big)^{-1}.$
Hence the fourth row follows.

At last suppose $(x,y)=(\frac{1}{2},\frac{\sqrt{3}}{2})$.
The automorphism where  $S:=\left(\begin{smallmatrix}1&-1\\0&-1\end{smallmatrix}\right)$ maps $(0,(\frac{k}{r},\frac{l}{r}),1)$ to $\left(0,(\frac{k-l}{r},-\frac{l}{r}),-1\right)=\left(0,(\frac{l-k}{r},\frac{l}{r}),1\right)^{-1}$.
Hence we get for each $\xi\in(\frac{1}{r}\mathds{Z},\frac{1}{r}\mathds{Z})$ an equivalent 
\begin{align*}
\xi_0\in&\left\{\big(\frac{k}{r},\frac{l}{r}\big)\mid k+l\leq r, 0<k<\frac{r}{2}<l<r, k\leq\frac{l}{2}\right\}\\
&\cup \left\{\big(\frac{k}{r},\frac{l}{r}\big)^T\mid 0\leq k\leq\frac{l}{2}\leq l\leq \frac{r}{2}\right\}.
\end{align*}
Additionally the automorphism where $S=\left(\begin{smallmatrix}1&0\\1&-1\end{smallmatrix}\right)$ maps $\big(0,(\frac{k}{r},\frac{l}{r}),1\big)$ to 
\[\left(\big(0,(0,-1),0\big)\big(0,(\frac{k}{r},\frac{r-l+k}{r}),1\big)\right)^{-1}.\]
Thus the last row follows.

\paragraph{Uniqueness}
Now we want to see that $\xi_0$ from the list is uniquely determined.
Therefor, we use a proof by contradiction.

Assume that for an $r\in\mathds{N}\backslash\{0\}$, $(x,y)\in\mathds{F}_1$ and a $\lambda=\lambda_0+k\pi$, where $\lambda_0=\left\{\frac{\pi}{3},\frac{\pi}{2},\frac{2\pi}{3},\pi\right\}$ and $k\in\mathds{N}$, there are $\xi$ and $\tilde\xi$ from the list, such that there is an automorphism $\varphi$ of $Osc_1(\omega_r,\lambda B_{x,y})$, mapping $L(\xi)$ onto $L(\tilde\xi)$.
Then, moreover, this automorphism maps $\Gamma_r$ onto itself.

So we can check each $\Gamma_r$-preserving automorphism and will note that none of them maps $L(\xi)$ onto $L(\tilde\xi)$, and we get our contradiction.\\

\noindent\hspace*{20mm}%
Let us begin with $\lambda=\pi+2\pi k$, $k\in\mathds{N}$ and $r$ even.
%Then for all $(z,\xi,t)\in Osc_1(\omega_1,\lambda B_{x,y})$ we get $(z,\xi,t)^{-1}=(-z,\xi,-t)$.
%Then an automorphism, mapping $\big\Gamma_r\cup\{(0,\xi,1)\} \big\rangle$ onto $\big\Gamma_r\cup\{(0,\tilde\xi,1)\}\big\rangle$, maps (0,\xi,1) to ((r,(s,t),0)(0,\tilde\xi,1))^{\pm 1}=(0,\eta,1)^{\pm 1}. Hence r\eta_{1/2}\is even iff \tilde\xi_{1/2} is even.  

%Again we just observe the second component.

Let $\varphi$ be an automorphism, given as in (\ref{AutOsc}) in Theorem \ref{theoremAutOsc}, which maps $L(\xi)$ onto $L(\tilde\xi)$.
Then the corresponding map \[\hat\varphi: \mathds{R}^2\rightarrow\mathds{R}^2,\quad\xi\mapsto S\xi+\e^Bb-b=S\xi-2b\] maps $\xi=(\xi_1,\xi_2)$ to a vector $\eta=(\eta_1,\eta_2)$, where $r\eta_{1/2}$ is even if and only if $\tilde\xi_{1/2}$ is even.

%Note that an automorphism mapping $\big\langle(1,0,0), (0,e_i,0), (0,\xi_1,1) \big\rangle$ onto $\big\langle(1,0,0),\linebreak (0,e_i,0), (0,\xi_2,1)\big\rangle$ also maps the first component of $\xi_1$ always onto a vector in $\mathds{R}^2$, having entries, whose $r$-times is even if and only if the $r$-times of the corresponding entry of $\xi_2$ is even.

But the corresponding map $\hat\varphi$ for an $\Gamma_r$-preserving automorphism maps $(0,0)$ to a vector in $\frac{2}{r}\mathds{Z}^2$.
Thus there is a contradiction for $\tilde\xi\in\left\{(\frac{1}{r},0),(0,\frac{1}{r}),(\frac{1}{r},\frac{1}{r})\right\}$.

In an analogue way, we find a contradiction if $\xi=(\frac{1}{r},0)$. But we have to subdivide the proof, depending on the value of $x$ and $y$.\\
If $x^2+y^2>1$ and $x\neq \frac{1}{2}$, then each $\Gamma_r$-preserving automorphism gives a corresponding map $\hat\varphi$, which maps $(\frac{1}{r},0)$ to a vector in $(\frac{2}{r}\mathds{Z}+\frac{1}{r},\frac{2}{r}\mathds{Z})$.
Thus there is a contradiction for all $\tilde\xi\in\left\{(0,\frac{1}{r}),(\frac{1}{r},\frac{1}{r})\right\}$.
If $x^2+y^2>1$ and $x=\frac{1}{2}$, then each $\Gamma_r$-preserving automorphism gives a corresponding map $\hat\varphi$, which maps $(\frac{1}{r},0)$ to a vector in $(\frac{1}{r}\mathds{Z},\frac{2}{r}\mathds{Z})$.
Thus there is a contradiction for $\tilde\xi=(\frac{1}{r},\frac{1}{r})$.\\
If $x^2+y^2=1$ and $x\notin\big\{0,\frac{1}{2}\big\}$, then the corresponding map $\hat\varphi$ for every $\Gamma_r$-preserving automorphism maps $(\frac{1}{r},0)$ to a vector with one entry in $\frac{2}{r}\mathds{Z}$ and one in $\frac{2}{r}\mathds{Z}+\frac{1}{r}$. 
Thus there is a contradiction for $\tilde\xi=(\frac{1}{r},\frac{1}{r})$.\\
At last suppose $\xi=(0,\frac{1}{r})$ for $x^2+y^2>1$ and $x\neq \frac{1}{2}$.
Each $\Gamma_r$-preserving automorphism gives a corresponding map $\hat\varphi$, which maps $(0,\frac{1}{r})$ to a vector in $(\frac{2}{r}\mathds{Z},\frac{1}{r}\mathds{Z})$. Thus there is a contradiction for $(\frac{1}{r},\frac{1}{r})$.

Altogether, we verified that the $\xi_0$ from the list, which we refer to a lattice  $L(\xi)$ in $Osc_1(\omega_r,\lambda B_{x,y})$, where $\lambda=\pi+2\pi k$, $k\in\mathds{N}$, $(x,y)\in\mathds{F}_1$, is uniquely determined.\\

\noindent\hspace*{20mm}%
Now, we consider the case that $\lambda=2\pi k$, $k\in \mathds{N}\backslash\{0\}$, and get the same contradiction.
But, first of all, we note that $\xi\in\mathds{R}^2$ and $\tilde\xi$ are equivalent if and only if there are $t_1,t_2\in\mathds{Z}$ and an integer matrix \[S\in\left\{\e^{tB},\e^{tB}\left(\begin{smallmatrix}1&-2x\\0&-1\end{smallmatrix}\right)\mid t\in\mathds{R}\right\}\] with $\det S=\mu$, such that $\tilde\xi=\mu S\xi+t_1e_1+t_2e_2$.

%But, first of all, we note that $\xi_0\in\mathds{R}^2$ and $aS\xi_0+t_1e_1+t_2e_2$ are equivalent, for each integer matrix \[S\in\left\{ \e^{tB},\e^{tB}\left(\begin{smallmatrix}1&-2x\\0&-1\end{smallmatrix}\right)\mid t\in\mathds{R}\right\}\] with $\det S=a$ and for each $t_1,t_2\in\mathds{Z}$.

So it suffices to fix a $t_1$ and $t_2$ for each integer matrix $S\in\left\{ \e^{tB},\e^{tB}\left(\begin{smallmatrix}1&-2x\\0&-1\end{smallmatrix}\right)\mid t\in\mathds{R}\right\}$ and each $\xi$, such that $\mu S\xi+t_1e_1+t_2e_2\in[0,1)^2$ and show that $\mu S\xi+t_1e_1+t_2e_2$ is equal to $\xi$ or not an element in the set from the list.

We consider the case that $(x,y)=(0,1)$.
Note that $\pm E_2$, $\pm \left(\begin{smallmatrix}0&-1\\1&0\end{smallmatrix}\right)$, $\pm \left(\begin{smallmatrix}0&1\\1&0\end{smallmatrix}\right)$ and $\pm \left(\begin{smallmatrix}1&0\\0&-1\end{smallmatrix}\right)$ are the only integer matrices in $\left\{ \e^{tB},\e^{tB}\left(\begin{smallmatrix}1&-2x\\0&-1\end{smallmatrix}\right)\mid t\in\mathds{R}\right\}$ for this case.

We will denote by $M$ the set $M:=\left\{\big(\frac{k'}{r},\frac{l'}{r}\big) | 0\leq k \leq l \leq \frac{r}{2} \right\}$ and set $\xi=(\frac{k}{r},\frac{l}{r})\in M$.
\begin{itemize}
\item If $S=E_2$, then $S\xi=\xi$.
\item Let $S=-E_2$.
For the cases that $k=l=0$, $k=l=\frac{r}{2}$, or $k=0$ and $l=\frac{r}{2}$, we get $S\xi=\xi$, $S\xi+e_1+e_2=\xi$ or $S\xi+e_1=\xi$ respectively.
If $0<k,l<\frac{r}{2}$, then $S\xi+e_1+e_2=(\frac{r-k}{r},\frac{r-l}{r})\in[0,1)^2$, but $\frac{r-k}{r}>\frac{r}{2}$.
At last, if $k=0$ and $0<l<\frac{r}{2}$, then $S\xi+e_1=(0,\frac{r-l}{r})\in[0,1)^2$, but $\frac{r-l}{r}>\frac{r}{2}$.
\item Let $S=\left(\begin{smallmatrix}0&-1\\1&0\end{smallmatrix}\right)$.
First of all, we see that $S\xi+e_1=\xi$, respectively $S\xi=\xi$ for $k=l=\frac{r}{2}$ or $k=l=0$.
If $l\notin\{0,\frac{r}{2}\}$, then $S\xi+e_1=(\frac{r-l}{r},\frac{k}{r})\in[0,1)^2$, but $\frac{r-l}{r}>\frac{1}{2}$.
If $l=\frac{r}{2}$ and $k<\frac{r}{2}$, then $S\xi+e_1=(\frac{1}{2},\frac{k}{r})\in[0,1)^2$, but $\frac{k}{r}<\frac{1}{2}$.
\item Similar arguments apply to the case $S=\left(\begin{smallmatrix}0&1\\-1&0\end{smallmatrix}\right)$.
\item Let $S=-\left(\begin{smallmatrix}0&1\\1&0\end{smallmatrix}\right)$.
It follows that $-S\xi=(\frac{l}{r},\frac{k}{r})\in[0,1)^2$.
Hence, we get $(\frac{l}{r},\frac{k}{r})\notin M$ for $k<l$, and $-S\xi=\xi$, for $k=l$.
\item Let $S=\left(\begin{smallmatrix}0&1\\1&0\end{smallmatrix}\right)$.
We see again that $-S\xi=\xi$ for $k=l=0$.
If $k=0$ and $l>0$, then $-S\xi+e_1\in[0,1)^2$, but $\frac{r-l}{r}>0$.
For $k,l\neq 0$ we get $-S\xi+e_1+e_2=(\frac{r-l}{r},\frac{r-k}{r})\in[0,1)^2$. If, additionally, $k<\frac{r}{2}$, then $\frac{r-k}{r}>\frac{r}{2}$.
If, however, $k=l=\frac{r}{2}$, then $(\frac{r-l}{r},\frac{r-k}{r})=\xi$.
\item Let $S=\left(\begin{smallmatrix}1&0\\0&-1\end{smallmatrix}\right)$.
If $k\notin\{0,\frac{r}{2}\}$, then $-S\xi+e_1=(\frac{r-k}{r},\frac{l}{r})\in[0,1)^2$, but $\frac{r-k}{r}>\frac{r}{2}$.
If $k=l=\frac{r}{2}$, then $-S\xi+e_1=\xi$ and if $k=0$, then $-S\xi=\xi$.
\item The same reasoning applies to the case $S=\left(\begin{smallmatrix}-1&0\\0&1\end{smallmatrix}\right)$.
\end{itemize} 
Finally, for $(x,y)=(0,1)$ it follows that $(\frac{k}{r},\frac{l}{r})\in M$ and $(\frac{k'}{r},\frac{l'}{r})\in M$ are equivalent, if and only if $k=k'$ and $l=l'$.

The rest of the case $\lambda=2\pi k$ runs as before.\\

\noindent\hspace*{20mm}%
For $\lambda=\lambda_0+k\pi$, where $\lambda_0\in\left\{\frac{\pi}{3},\frac{\pi}{2},\frac{2\pi}{3}\right\}$ and $k\in\mathds{N}$, we use some other way to prove the assertion.
First of all we begin with a definition.
\begin{definition}\label{OG}
A group, generated by four elements $\{\alpha,\beta,\gamma,\delta\}$ is called O-lattice, if:
\begin{itemize}
\item There is an $r\in\mathds{N}\backslash\{0\}$, such that $\alpha\beta\alpha^{-1}\beta^{-1}=\gamma^r$.
\item There is a $k\in\mathds{N}\backslash\{0\}$, such that $\delta^k$ and $\gamma$ generate the center of the group and
\item $\delta\alpha\delta^{-1}$ and $\delta\beta\delta^{-1}$ are both elements of $\langle\alpha,\beta,\gamma\rangle$.
\end{itemize}
\end{definition}

It is not hard to see that the computed lattices of $Osc_1(\omega_r,\lambda B_{x,y})$, where $\lambda\neq k\pi$, $k\in\mathds{N}$ satisfy Definition \ref{OG}.

%Standard computation and the fact that each element of an O-lattice can be written as $\alpha^a\beta^b\gamma^c\delta^d$ for some $a,b,c,d\in\mathds{Z}$ brings the following lemma:
Standard arguments yield the following lemma.
\begin{lemma}
Let $G$ be an O-lattice as in Definition \ref{OG}, $H$ a group and $\tilde{\varphi}:\left\{\alpha,\beta,\gamma,\delta\right\}\rightarrow H$ a map.
The map $\varphi:G\rightarrow H$, defined by \[\varphi(\alpha^x\beta^y\gamma^z\delta^t)=\tilde{\varphi}(\alpha)^x\tilde{\varphi}(\beta)^y\tilde{\varphi}(\gamma)^z\tilde{\varphi}(\delta)^t\]
for all $x,y,z,t\in\mathds{Z}$, is a homomorphism, if and only if:
\begin{itemize}
\item $\varphi(\delta)\varphi(\alpha)\varphi(\delta)^{-1}=\varphi(\delta\alpha\delta^{-1})$,
\item $\varphi(\delta)\varphi(\beta)\varphi(\delta)^{-1}=\varphi(\delta\beta\delta^{-1})$,
\item $\varphi(\gamma)$ is an element of the center of $H$ and
\item $\varphi(\alpha)\varphi(\beta)\varphi(\alpha)^{-1}\varphi(\beta)^{-1}=\varphi(\gamma)^r$.
\end{itemize}
\end{lemma}

\begin{remark}\label{kgleichl}
Let $G$ be an O-lattice as in Definition \ref{OG}.
Let, furthermore, $H$ be an O-lattice, generated by $\big\{\tilde{\alpha},\tilde{\beta},\tilde{\gamma},\tilde{\delta}\big\}$, where $\tilde{\gamma}$ and $\tilde{\delta}^l$ generate the center of $H$ and $\tilde{\alpha}\tilde{\beta}\tilde{\alpha}^{-1}\tilde{\beta}^{-1}=\tilde{\gamma}^s$.
Let $\varphi: G\rightarrow H$ be an isomorphism.
Then $\varphi(\gamma)=\gamma^{\pm 1}$. To see this, note that $\varphi(\gamma)\in\langle\tilde\gamma\rangle$, since $\gamma$ is in the center of $G$ and $\gamma^r$ in the commutator subgroup. Then using the bijectivity of $\varphi$ yields the assertion.
Furthermore, $\varphi(\delta^k)=\tilde\gamma^p\tilde\delta^{\pm l}$ for some $p\in\mathds{Z}$.
Indeed, since $\varphi$ is an isomorphism, there are $p,q,u,v\in\mathds{Z}$, such that $\varphi(\delta^k)=\tilde\gamma^p\tilde\delta^{ul}$ and $\varphi(\gamma^q\delta^{vk})=\tilde\delta^{l}$.
So we get $\delta^k=\varphi^{-1}(\varphi(\delta^k))=\gamma^{qu\pm p}\delta^{uvk}$
and the assertion holds.
In addition, it is straight forward to see that $\varphi$ maps $\delta$ to $\tilde{\alpha}^x\tilde{\beta}^y\tilde{\gamma}^z\tilde{\delta}^{\pm 1}$ for some $x,y,z\in\mathds{Z}$ and that $k=l$.
%
%If $G$ and $H$ are isomorphic by $\varphi$, then there is an isomorphism from $G/\langle\gamma^r\rangle$ to $H/\langle\varphi(\gamma)^r\rangle$.
%We see that $[\gamma]$ and $[\tilde\gamma]$ have finite order and $[\delta^k]$ and $[\tilde\delta^l]$ infinite one. 
%Since the center of $G/\langle\gamma^r\rangle$ and the center of $H/\langle\varphi(\gamma)^r\rangle$ are isomorphic, $\varphi$ maps $\delta^k$ onto $\tilde{\gamma}^t\tilde{\delta}^{\pm l}$ for some $t\in\mathds{Z}$ and $\gamma$ onto $\tilde{\gamma}^{\pm 1}$.
%
\end{remark}

\begin{lemma} \label{nichtIsom}
For a fixed $B=\lambda B_{0,1}$, where $\cos\lambda=0$ and an even $r$,
the lattices $L(\xi)$ and $L(\tilde\xi)$ of $Osc_1(\omega_r,B)$, where $\xi\neq\tilde\xi$ from the list, are not isomorphic as abstract groups.
\end{lemma}

\begin{proof}
We first prove the assertion for $\lambda=\frac{1}{2}\pi+2\pi k$ and $k\in\mathds{N}$.
%Den Fall $\lambda=\frac{3}{2}\pi+2\pi k$, $k\in\mathds{N}$ werden wir anschließend auf diesen zurückführen.
%Sei $B=\lambda B_{0,1}$ mit $\lambda=\frac{1}{2}\pi+2\pi k$, $k\in\mathds{N}$ und $r$ gerade.
We set $\alpha:=(0,e_1,0)$, $\beta:=(0,e_2,0)$, $\gamma:=(1,0,0)$, $\delta_0:=(0,0,1)$ and $\delta_1:=\big(0,(\frac{1}{r},0),1\big)$.
Using equation (\ref{LCapHeis}) and $\e^B=\left(\begin{smallmatrix}0&-1\\1&0\end{smallmatrix}\right)$ we get
\begin{align*}
\delta_0\alpha\delta_0^{-1}=\beta,\quad \delta_0\beta\delta_0^{-1}=\alpha^{-1},\quad \delta_1\alpha\delta_1^{-1}=\beta\gamma,\quad \delta_1\beta\delta_1^{-1}=\alpha^{-1}.
\end{align*}

Suppose that there is an isomorphism $\varphi$ from $\langle\alpha,\beta,\gamma,\delta_0\rangle$ onto $\langle\alpha,\beta,\gamma,\delta_1\rangle$.
%Then $\delta_0$ is mapped onto $\alpha^x\beta^y\gamma^z\delta_1^{\pm 1}$ for some $x,y,z\in\mathds{Z}$ and it follows that
Then 
$\varphi(\delta_0)=\alpha^x\beta^y\gamma^z\delta_1^{\pm 1}$ for some $x,y,z\in\mathds{Z}$.
Furthermore $\varphi(\alpha)=\alpha^{n_1}\beta^{n_2}\gamma^{n_3}$ for some $n_1,n_2,n_3\in\mathds{Z}$, since $\varphi(\delta_0)^2\varphi(\alpha)\varphi(\delta_0)^{-2}=\varphi(\alpha)^{-1}$ and 
$\delta_1\eta\delta_1^{-1}\in\langle\alpha,\beta,\gamma\rangle$ for every $\eta\in\langle\alpha,\beta,\gamma\rangle$.
On the same way, we get that $\varphi(\beta)=\alpha^{n'_1}\beta^{n'_2}\gamma^{n'_3}$ for some $n'_1,n'_2,n'_3\in\mathds{Z}$.

Let $\varphi(\delta_0)=\alpha^x\beta^y\gamma^z\delta_1^{1}$ (similar arguments apply to the case $\varphi(\delta_0)=\alpha^x\beta^y\gamma^z\delta_1^{-1}$).
Then 
\begin{align*}
\alpha^{n'_1}\beta^{n'_2}\gamma^{n'_3}=\varphi(\beta)&=\varphi(\delta_0)\varphi(\alpha)\varphi(\delta_0)^{-1}\\
%&=\alpha^x\beta^y\gamma^z\delta_1\alpha^{n_1}\beta^{n_2}\gamma^{n_3}\delta_1^{-1}\gamma^{-z}\beta^{-y}\alpha^{-x}\\
%&=\alpha^x\beta^y\beta^{n_1}\alpha^{-n_2}\beta^{-y}\alpha^{-x}\gamma^{n_1+n_3}\\
&=\alpha^{-n_2}\beta^{n_1}\gamma^{n_1+n_3+r\mathcal{O}_1},
\end{align*}
for some $\mathcal{O}_1\in\mathds{Z}$.
Hence $n'_1=-n_2$, $n'_2=n_1$ and $n'_3=n_1+n_3+r\mathcal{O}_1$.
Similar computation for $\varphi(\delta_0)\varphi(\beta)\varphi(\delta_0)^{-1}$ shows that $-n_3=n'_1+n'_3+r\mathcal{O}_2$, for some $\mathcal{O}_2\in\mathds{Z}$.
Hence $-n_2+n'_3+n_3$ and $n_1+n_3-n'_3$ are even. 
Thus $n_1+n_2$ is even. 
%$\varphi(\delta_0)^2\varphi(\alpha)\varphi(\delta_0)^{-2}=\varphi(\alpha)^{-1}$.
%Furthermore, $\varphi(\alpha)=\alpha^{m_1}\beta^{m_2}\gamma^{m_3}$ for some $m_1,m_2,m_3\in\mathds{Z}$ and  $\varphi(\beta)=\alpha^{m'_1}\beta^{m'_2}\gamma^{m'_3}$ for some $m'_1,m'_2,m'_3\in\mathds{Z}$, since $\delta_1\eta\delta_1^{-1}\in\langle\alpha,\beta,\gamma\rangle$ for every $\eta\in\langle\alpha,\beta,\gamma\rangle$.
%Let, at first, $\varphi(\delta_0)=\alpha^x\beta^y\gamma^z\delta_1$ and $\varphi(\alpha)=\alpha^{m_1}\beta^{m_2}\gamma^{m_3}$ for some $m_1,m_2,m_3\in\mathds{Z}$
%(similar arguments apply to the case $\varphi(\delta_0)=\alpha^x\beta^y\gamma^z\delta_1^{-1}$).
%Then 
%\begin{align*}
%&\gamma^{-m_3}\beta^{-m_2}\alpha^{-m_1}=\varphi(\alpha)^{-1}=\varphi(\delta_0)^2\varphi(\alpha)\varphi(\delta_0)^{-2}\\
%&=\gamma^{-rxm_2+rym_2+rm_1m_2+rm_1x+rm_1y+m_1-m_2+m_3}\beta^{-m_2}\alpha^{-m_1}\gamma^{m_1-m_2+m_3}.
%\end{align*}
%Thus, $2m_3=rxm_2-rym_2-rm_1m_2-rm_1x-rm_1x-rm_1y-m_1+m_2$.
%Similar computation for $\varphi(\beta)=\alpha^{m'_1}\beta^{m'_2}\gamma^{m'_3}$ shows that $2m'_3=rxm'_2-rym'_2-rm'_1m'_2-rm'_1x-rm'_1x-rm'_1y-m'_1+m'_2$.
%Hence $m_3$ and $m'_3$ are integers, so $m_2-m_1$ and $m'_2-m'_1$ have to be even. 
This contradicts the bijectivity of $\varphi$.
So there is no isomorphism.

We consider the second case now.
For $B=\lambda B_{0,1}$, $\lambda=\frac{3}{2}\pi+2\pi k$, $k\in\mathds{N}$ and $r$ even we set $\tilde\alpha:=(0,e_1,0), \tilde\beta:=(0,e_2,0), \tilde\gamma:=(1,0,0)$, $\tilde\delta_0:=(0,0,1)$ and $\tilde\delta_1:=\big(0,(0,\frac{1}{r}),1\big)$.

The maps defined by
\begin{align*}
\tilde\alpha\mapsto\alpha^{-1},\quad\tilde\beta\mapsto\beta^{-1},\quad\tilde\gamma\mapsto\gamma,\quad\tilde\delta_i\mapsto\delta_i^{-1}
\end{align*}
are isomorphisms from the O-lattice $\big\langle\tilde\alpha,\tilde\beta,\tilde\gamma,\tilde\delta_i\big\rangle$ onto $\big\langle\alpha,\beta,\gamma,\delta_i\big\rangle$ for $i=0,1$.
Hence, the lemma follows, by using the first case of the proof.
\end{proof}

\begin{lemma}\label{Lr2}
For a fixed $B=\lambda B_{\frac{1}{2},\frac{\sqrt{3}}{2}}$, where $\cos\lambda=-\frac{1}{2}$ and an $r$ divisible by $6$, the lattices $L(\xi)$ and $L(\tilde\xi)$ of $Osc_1(\omega_r,B)$, where $\xi\neq\tilde\xi$ from the list, are not isomorphic as abstract groups.
\end{lemma}

\begin{proof}
We subdivide the proof into two parts and argue as in the proof of the previous lemma.

At first let $B=\lambda B_{\frac{1}{2},\frac{\sqrt{3}}{2}}$, where $\cos\lambda=-\frac{1}{2}$, $\sin\lambda=\frac{\sqrt{3}}{2}$ and let $r$ be divisible by $6$.
We set $\alpha:=(0,e_1,0), \beta:=(0,e_2,0), \gamma:=(1,0,0), \delta_0:=(0,0,1)$ and $\delta_1:=\big(0,(\frac{1}{r},0),1\big)$.
Using equation (\ref{LCapHeis}) and $\e^{B}=\left(\begin{smallmatrix}0&-1\\1&-1\end{smallmatrix}\right)$ gives:
\begin{align*}
\delta_0\alpha\delta_0^{-1}=\beta,\quad
\delta_0\beta\delta_0^{-1}=\alpha^{-1}\beta^{-1}\gamma^{-\frac{r}{2}},\quad\delta_1\alpha\delta_1^{-1}=\beta\gamma,\quad
\delta_1\beta\delta_1^{-1}=\alpha^{-1}\beta^{-1}\gamma^{-\frac{r}{2}-1}.
\end{align*}

Suppose that there is an isomorphism $\varphi$ from $\langle\alpha,\beta,\gamma,\delta_0\rangle$ onto $\langle\alpha,\beta,\gamma,\delta_1\rangle$.

Then $\varphi(\delta_0)\varphi(\alpha)\varphi(\delta_0)^{-1}=\varphi(\delta_0\alpha\delta_0^{-1})=\varphi(\beta).$\\
Let $\varphi(\delta_0)=\alpha^x\beta^y\gamma^z\delta_1$ (The case $\varphi(\delta_0)=\alpha^x\beta^y\gamma^z\delta_1^{-1}$ runs similar).\\

Since $\delta_0\eta\delta_0^{-1}\in\langle\alpha,\beta,\gamma\rangle$ for all $\eta\in\langle\alpha,\beta,\gamma\rangle$, one can check that $\varphi(\alpha)=\alpha^{m_1}\beta^{m_2}\gamma^{m_3}$ for some $m_1,m_2,m_3\in\mathds{Z}$.

Thus
\begin{align*}
\varphi(\beta)&=\varphi(\delta_0)\varphi(\alpha)\varphi(\delta_0^{-1})\\
	&=\alpha^{-m_2}\beta^{m_1-m_2}\gamma^{m_1+m_2(-r/2-1)+m_3+r\mathcal{O}_1},
\end{align*}
where $\mathcal{O}_1\in\mathds{Z}$.

Furthermore, we obtain
\begin{align*}
\gamma^{\pm r}&=\varphi(\alpha)\varphi(\beta)\varphi(\alpha)^{-1}\varphi(\beta)^{-1}\\
&=\gamma^{r(m_1^2-m_1m_2+m_2^2)}.
\end{align*}
Hence, $\varphi(\gamma)=\gamma$ and $(m_1,m_2)\in\big\{(1,0),(-1,0),(0,1),(0,-1),(1,1),(-1,-1)\big\}$.

In addition
\begin{align*}
&\alpha^{m_2-m_1}\beta^{-m_1}\gamma^{-m_1-m_2(-\frac{r}{2}-1)-2m_3-\frac{r}{2}+r\mathcal{O}_3}\\
&=\varphi(\alpha)^{-1}\varphi(\beta)^{-1}\varphi(\gamma)^{-\frac{r}{2}}\\
&=\varphi(\delta_0)\varphi(\beta)\varphi(\delta_0)^{-1}\\
&=\alpha^x\beta^y\gamma^z\alpha^{-m_2}\beta^{m_1-m_2}\gamma^{m_1+m_2(-\frac{r}{2}-1)+m_3+r\mathcal{O}_1}\delta_1^{-1}\beta^{-y}\alpha^{-x}\\
&=\alpha^{m_2-m_1}\beta^{-m_1}\gamma^{-m_2-\frac{r}{2}m_1+m_3+r\mathcal{O}_2},
\end{align*}
where $\mathcal{O}_2,\mathcal{O}_3\in\mathds{Z}$.
Thus $3m_3=\frac{r}{2}m_1-m_1+\frac{r}{2}m_2+2m_2-\frac{r}{2}+r\mathcal{O}_4$, for some $\mathcal{O}_4\in\mathds{Z}$.
Since $m_3\in\mathds{Z}$, the term $-m_1+2m_2$ must be divisible by $3$.

But, contrarily, for $(m_1,m_2)\in\big\{(1,0),(0,1),(-1,0),(0,-1),(1,1),(-1,-1)\big\}$ it follows that $2m_2-m_1$ is not divisible by $3$.
Finally, the first case is shown.

For the second case let $\sin\lambda=-\frac{\sqrt{3}}{2}$.
We set $\tilde\alpha:=(0,e_1,0)$, $\tilde\beta:=(0,e_2,0)$, $\tilde\gamma:=(1,0,0)$, $\tilde\delta_0:=(0,0,1)$ and $\tilde\delta_1=\big(0,(0,\frac{1}{r}),1\big)$.
The maps defined by
\begin{align*}
\tilde\alpha\mapsto\beta,\quad\tilde\beta\mapsto\alpha,\quad\tilde\gamma\mapsto\gamma^{-1},\quad\tilde\delta_i\mapsto\delta_i
\end{align*}
are isomorphisms from the O-lattice $\big\{\tilde\alpha,\tilde\beta,\tilde\gamma,\tilde\delta_i\big\}$ onto 
$\{\alpha,\beta,\gamma,\delta_i\}$ for $i=0,1$.
Using the first part brings the assertion.
\end{proof}
%
%By the same method, we can proof this lemma.

\begin{lemma}
For a fixed $B=\lambda B_{\frac{1}{2},\frac{\sqrt{3}}{2}}$ where $\cos\lambda=-\frac{1}{2}$ and an odd $r$ divisible by $3$, the lattices $L(\xi)$ and $L(\tilde\xi)$ of $Osc_1(\omega_r,B)$, where $\xi\neq\tilde\xi$ from the list, are not isomorphic as abstract groups.
\end{lemma}

\begin{proof}
This follows by the same method as in Lemma \ref{Lr2}.
\end{proof}

Finally, Theorem \ref{Klassifik} is verified.
\end{proof}

\renewcommand{\arraystretch}{1.4} 
\begin{appendix}
\section{List of $\xi_0$}\label{AppA}
\begin{tabular}{|c|c|c|c|} \hline	

$\lambda$		& $x,y$ 			& $\xi_0$ for an even $r$ 		& $\xi_0$ for an odd $r$ 		\\
\hline 
\hline
%1
$\frac{1}{3}\pi+2\pi k$ & $x=\frac{1}{2},y=\frac{\sqrt{3}}{2}$ & $\{(0,0)\}$ & $\{(\frac{1}{2r},0)\}$ \\ \hline

%2
$\frac{5}{3}\pi+2\pi k$ & $x=\frac{1}{2},y=\frac{\sqrt{3}}{2}$ & $\{(0,0)\}$ & $\{(0,\frac{1}{2r})\}$ \\ \hline

%3
$\frac{1}{2}\pi+2\pi k$ & $x=0,y=1$ & $\{(0,0),(0,\frac{1}{r})\}$ & $\{(0,0)\}$ \\ \hline

%4
$\frac{3}{2}\pi+2\pi k$ & $x=0,y=1$ & $\{(0,0),(0,\frac{1}{r})\}$ & $\{(0,0)\}$ \\ \hline

%5
\multirow{2}{*}{$\frac{2}{3}\pi+ 2\pi k$} & \multirow{2}{*}{$x=\frac{1}{2},y=\frac{\sqrt{3}}{2}$} & $r \equiv 0(3)$: $\{(0,0),(\frac{1}{r},0)\}$ & $r \equiv 0(3)$ : $\{(0,\frac{1}{2r}),(\frac{1}{r},\frac{1}{2r})\}$ \\ 

& & else: $\{(0,0)\}$& else: $\{(0,\frac{1}{2r})\}$\\ \hline

%6
\multirow{2}{*}{$\frac{4}{3}\pi+ 2\pi k$} & \multirow{2}{*}{$x=\frac{1}{2},y=\frac{\sqrt{3}}{2}$} & $r \equiv 0(3)$: $\{(0,0),(\frac{1}{r},0)\}$ & $r \equiv 0(3)$: $\{(\frac{1}{2r},0),(\frac{1}{2r}+\frac{1}{r},0)\}$ \\ 
& & else: $\{(0,0)\}$& else: $\{(\frac{1}{2r},0)\}$\\ \hline

%7
\multirow{4}{*}{$\pi+2\pi k$}	&	 $x^2+y^2>1,\,$ 		& 	\multirow{2}{*}{$\{(0,0),(\frac{1}{r},0),(0,\frac{1}{r}),(\frac{1}{r},\frac{1}{r})\}$} 	&	 \multirow{7}{*}{$\{(0,0)\}$} \\
					&	$x\neq\frac{1}{2 }$	&&\\  \cline{2-3}

%8
& $x^2+y^2>1,$ & \multirow{2}{*}{$\{(0,0),(\frac{1}{r},0),(\frac{1}{r},\frac{1}{r})\}$} & \\  
&$x=\frac{1}{2}$ && \\ \cline{2-3}

%8,5
& $x^2+y^2=1,$  & \multirow{2}{*}{$\{(0,0),(\frac{1}{r},0),(\frac{1}{r},\frac{1}{r})\}$} & \\ 
& $x\neq\frac{1}{2}$ & &\\ \cline{2-3}

%9
& $x=\frac{1}{2},y=\frac{\sqrt{3}}{2}$ & $\{(0,0),(\frac{1}{r},\frac{1}{r})\}$ &	\\ \hline

%10
\multirow{11}{*}{$2\pi k$} & $x^2+y^2>1$ & \multicolumn{2}{|c|}{\multirow{2}{*}{$\{(\frac{k}{r},\frac{l}{r}) | 0\leq k,l\leq \frac{r}{2}\}\cup \{(\frac{k}{r},\frac{l}{r})| 0<k<\frac{r}{2}<l<r \}$}} \\ 

& $x\notin\{0,\frac{1}{2}\}$ & \multicolumn{2}{|c|}{} \\ \cline{2-4}

%11
& $x^2+y^2>1$ & \multicolumn{2}{|c|}{\multirow{2}{*}{$\{(\frac{k}{r},\frac{l}{r})\mid 0\leq k, l\leq \frac{r}{2}\}$}} \\ 
 & $x=0$ &  \multicolumn{2}{|c|}{ } \\ \cline{2-4}

%12
& $x^2+y^2>1$ & \multicolumn{2}{|c|}{$\{(\frac{k}{r},\frac{l}{r})\mid 0\leq ,\,k\leq\frac{l}{2}\leq l\leq \frac{r}{2}\}\cup \{(\frac{k}{r},0)\mid 0\leq k\leq\frac{r}{2}\}\cup$} \\ 

& $x=\frac{1}{2}$ &\multicolumn{2}{|c|}{$\{(\frac{k}{r},\frac{l}{r})\mid <k<\frac{r}{2}<l<r,\,k\leq\frac{l}{2}\}$}\\ \cline{2-4}

%13
& $x^2+y^2=1$ & \multicolumn{2}{|c|}{\multirow{2}{*}{$\{(\frac{k}{r},\frac{l}{r}) | 0\leq k\leq l\leq \frac{r}{2}\}\cup \{(\frac{k}{r},\frac{l}{r}) | 0< k<\frac{r}{2}<l<r, k+l\leq r\}$}} \\ 

&  $x\notin\{0,\frac{1}{2}\}$  &  \multicolumn{2}{|c|}{}  \\ \cline{2-4}

%14
& $x=0,y=1$ & \multicolumn{2}{|c|}{$\{(\frac{k}{r},\frac{l}{r}) | 0\leq k \leq l \leq \frac{r}{2} \}$} \\ \cline{2-4}

%15
& $x=\frac{1}{2},y=\frac{\sqrt{3}}{2}$ & \multicolumn{2}{|c|}{$\{(\frac{k}{r},\frac{l}{r}) | 0 \leq 2k\leq l\leq \frac{k+r}{2}\}$} \\ \hline

\end{tabular}
\end{appendix}

% %%%%%%%%%%%%%% B I B L I O G R A P H I E
% 
%\bibliography{/PFAD/ZU/DER/DATEI/diss.bib} % Trage hier den Pfad zu deiner diss.bib-Datei ein
%\backmatter % hiermit wird der Nachspann eingeleitet. Seitenzahlen werden weiter fortgeführt
%\bibliographystyle{plain}
\linespread{1.25}

%\nocite{*} %nicht benutzte Referenzen werden mit angegeben
\printbibliography

%\addbibresource{Literatur.bib}

%\addcontentsline{toc}{section}{Literaturverzeichnis}

\end{document}